\documentclass[12pt, leqno]{amsart}
\usepackage[all]{xy}
\usepackage{amsthm,amsfonts,amsmath,oldgerm,amssymb,amscd,inputenc}
\usepackage{mathrsfs}
\usepackage{comment} 

\usepackage{todonotes}

\allowdisplaybreaks[4] 
\newtheoremstyle{myremark}     {10pt}{10pt}{}{}{\bfseries}{.}{.5em}{}

\newtheorem{theorem}{Theorem}
\newcounter{tmp}
\newtheorem{proposition}[theorem]{Proposition}
\newtheorem{lemma}[theorem]{Lemma}
\newtheorem{coro}[theorem]{Corollary}

\newtheorem{remark}[theorem]{Remark}

\newtheorem*{theorem*}{Theorem}
\newtheorem*{proposition*}{Proposition}
\newtheorem{question}{Question}

\newcommand{\Z}{\mbox{$\mathbb Z$}}

\newcommand{\N}{\mbox{$\mathbb N$}}     
\newcommand{\R}{\mbox{$\mathbb R$}}     
\newcommand{\C}{\mbox{$\mathbb C$}}     

\newcommand{\norm}[1]{\left\Vert#1\right\Vert}
\begin{document}
	
	\title{Poisson boundary on full Fock space}
	
	\author[Bhat]{B.V. Rajarama Bhat}
	\address{Indian Statistical Institute, Stat-Math Unit, R V College Post, Bengaluru 560059, India}
	\email{bhat@isibang.ac.in}

	\author[Bikram]{Panchugopal Bikram}
	\address{School of Mathematical Sciences,
		National Institute of Science Education and Research,  Bhubaneswar, An OCC of Homi Bhabha National Institute,  Jatni- 752050, India}
	\email{bikram@niser.ac.in}

	\author[De]{Sandipan De}
	\address{Sandipan De, School of Mathematics and Computer Science, Indian Institute of Technology Goa,
		Farmagudi, Ponda-403401, Goa, India.}
	\email{sandipan@iitgoa.ac.in, 444sandipan@gmail.com}

	\author[Rakshit]{ Narayan Rakshit}
	\address{Indian Statistical Institute, Stat-Math Unit, R V College Post, Bengaluru 560059, India}
	\email{narayan753@gmail.com}
	
	\maketitle
	
	\newenvironment{dedication}{}{}
	\begin{dedication}
		\begin{center}
			\textit{Dedicated to Prof. V.S. Sunder}
		\end{center}
	\end{dedication}

	\begin{abstract} 
		This article is devoted to studying the non-commutative Poisson boundary associated with $\Big(B\big(\mathcal{F}(\mathcal{H})\big), P_{\omega}\Big)$
		where $\mathcal{H}$ is a separable Hilbert space (finite or infinite-dimensional), $\dim \mathcal{H} > 1$, with an orthonormal basis $\mathcal{E}$, $B\big(\mathcal{F}(\mathcal{H})\big)$ is the algebra of bounded linear operators on the full Fock space $\mathcal{F}(\mathcal{H})$ defined over $\mathcal{H}$, $\omega = \{\omega_e : e \in \mathcal{E} \}$ is a sequence of strictly positive real numbers such that $\sum_e \omega_e = 1$ and $P_{\omega}$ is the Markov operator  on $B\big(\mathcal{F}(\mathcal{H})\big)$ defined by
		\begin{align*}
			P_{\omega}(x) = \sum_{e \in \mathcal{E}} \omega_e l_e^* x l_e, \  x \in B\big(\mathcal{F}(\mathcal{H})\big),
		\end{align*}
		where, for $e \in \mathcal{E}$, $l_e$ denotes the left creation operator associated with $e$.
		We observe that the non-commutative Poisson boundary associated with $\Big(B\big(\mathcal{F}(\mathcal{H})\big), P_{\omega}\Big)$ is $\sigma$-weak closure of the Cuntz algebra $\mathcal{O}_{\dim \mathcal{H}}$ generated by the right creation operators. We prove that the Poisson boundary is an injective factor of type $III$ for any choice of $\omega$. Moreover, if $\mathcal{H}$ is finite-dimensional, we completely classify the Poisson boundary in terms of	its Connes' $S$ invarinat and curiously they are type $III _{\lambda }$ factors with $\lambda$ belonging to a certain small class of algebraic numbers.

	\end{abstract}

	
	\section{introduction}
	This article is dedicated to studying the non-commutative Poisson boundary associated to a certain unital completely positive (henceforth, abbreviated UCP) normal map on the algebra of bounded linear operators on the full Fock space over a separable Hilbert space.
	Given a  normal UCP map on a von Neumann algebra to itself, one can equip the fixed point set with an abstract von Neumann algebra structure called non-commutative Poisson boundary.

	Let us elucidate the notion of non-commutative Poisson boundary in more detail. Given a von Neumann algebra $N$ and an operator system $L$ in $N$ (that is, $L$ is a self-adjoint linear subspace of $N$ containing the identity), it is known that (see \cite[Theorem 3.1]{ChEf77}, \cite[Theorem 2.6]{Arv}) if there exists a completely positive projection
	$E: N \rightarrow N$ with image $E(N) = L$, then $L$ becomes a $C^*$-algebra with respect to the multiplication given by 
	\begin{align*}
		x \circ y = E(xy), \ x, y \in L,
	\end{align*}
	(which we call the  Choi-Effros product).

	Let $P$ be a normal UCP map from $N$ to itself. Such a map is called a non-commutative Markov operator. An element $x \in N$ is said to be $P$-harmonic if $P(x) = x$. We denote by $H^{\infty}(N, P)$ the set of all $P$-harmonic elements of $N$, that is,
	\begin{align*}
		H^{\infty}(N, P) = \{x \in N: P(x) = x\}.
	\end{align*}
	Then $H^{\infty}(N, P)$ is a $\sigma$-weakly closed operator system and it is the image of a completely positive projection of $N$. Indeed, Izumi showed in \cite[Theorem 3.3]{Iz2} that if we choose
	a free ultrafilter $\kappa \in \beta \mathbb{N} \setminus \mathbb{N}$, and define $E : N \rightarrow H^{\infty}(N, P)$ by
	\begin{align*}
		E(x) = \lim_{n \rightarrow \kappa} \frac{1}{n}\sum_{k = 0}^{n-1} P^k(x), \ x \in N,
	\end{align*}
	where the limit is taken in the weak operator topology, then $E$ is the desired projection and the Choi-Effros product equips $H^{\infty}(N, P)$ with a $C^*$-algebra structure. As $H^{\infty}(N, P)$ is a $\sigma$-weakly closed operator system, it is isometrically isomorphic to the dual of a Banach space and hence, from a theorem of Sakai \cite{Sakai} it follows that the $C^*$-algebra $H^{\infty}(N, P)$ can be represented faithfully as a von Neumann algebra. We call the resulting von Neumann algebra the non-commutative Poisson boundary of $(N, P)$. Although $E$ depends on the choice of the free ultrafilter $\kappa$, the Choi-Effros product of $H^{\infty}(N, P)$ does not depend on it,  because an operator system may have at most one von Neumann algebra structure.

	
	It was pointed out by W. Arveson that the non-commutative Poisson
	boundary for $P$ is identified with the fixed point algebra of the minimal dilation of $P$. To be more precise, let $(M, \alpha, p)$ denote the minimal dilation of $(N, P)$ where $M$ is a von Neumann algebra, $p$ is a projection in $M$ such that the central carrier of $p$ is $1_M$, and $\alpha$ is a unital normal endomorphism of $M$ such that $N = pMp$, $M$ is generated by $\bigcup_{n \geq 0} \alpha^n(N)$, and $P^n(a) = p \alpha^n(a) p$ for all $a \in N$ and $n \geq 1$. Izumi proved in \cite[Theorem 5.1]{Iz3} that the map 
	\begin{align*}
		&\theta: x \in M^{\alpha}:= \{x \in M: \alpha(x) = x\} \mapsto pxp \in H^{\infty}(N, P)
	\end{align*}
	is a completely positive order isomorphism between the two operator systems. In particular, the von Neumann algebra
	$M^{\alpha}$ gives a realization of the von Neumann algebra structure of $H^{\infty}(N, P)$. One of the useful consequences of this dilation theoretic approach is the following result \cite[Corollary 5.2]{Iz3} which we shall use frequently in the sequel to compute the Choi-Effros product of elements of the Poisson boundary.
	\begin{lemma}\label{formula}
		For any $a, b \in H^{\infty}(N, P)$, the sequence $\{P^n(ab)\}$ converges to the Choi-Effros product $a \circ b$ in the strong operator topology. 
	\end{lemma}

	Poisson boundaries over discrete quantum groups were first studied by Izumi \cite{Iz1}, in particular for the dual of Woronowicz's compact quantum group $SU_q(2)$. Izumi's result was generalized to the case of $SU_q(n)$ by Izumi, Neshveyev and Tuset \cite{IzNeTu}. Poisson boundaries for other discrete quantum groups have been studied by Vaes, Vander Vennet and Vergnioux \cite{VaVe08}, \cite{VaVe10}, \cite{VaVer07}.
	In general for a given  Markov operator $P$ on a von Neumann algebra $N$, it is a hard problem to find a concrete realization of the von Neumann algebra $H^\infty(N, P)$, even in the commutative case.  Kaimanovich refers to it as an identification problem \cite{Kia}. In \cite[Theorem 4.1]{Iz2} Izumi showed that if $\Gamma$ is a discrete countable group with a probability measure $\mu$ on $\Gamma$, $\rho$ denotes the right regular representation of $\Gamma$ on $\ell^2(\Gamma)$, and $Q_\mu$ is the Markov operator on $B(\ell^2(\Gamma))$ defined by  
	$$  Q_\mu(x) = \sum_{\gamma \in \Gamma} \mu(\gamma) \rho(\gamma) x \rho(\gamma^{-1}),~~ x \in B(\ell^2(\Gamma)),$$ then the Poisson boundary of
	$(B(\ell^2(\Gamma)), Q_\mu)$ is isomorphic to the crossed product of the Poisson boundary on the level of $\ell^\infty(\Gamma)$ with the canonical action of $\Gamma$ on it. Izumi then raised the question \cite[Problem 4.3]{Iz2} if such an identification result holds for a general second countable group $\Gamma$ with a reasonable good probability measure on it. This question was answered affirmatively by Jaworski and Neufang in \cite{JaNe07}.  This result was further generalized in \cite{KaNeRu} for locally compact quantum groups.
	
	An additional motivation for this article stems from the following example \cite[Page 360]{Iz3}. 
	Let $\mathcal{H}$ denote a $1$-dimensional Hilbert space with an orthonormal basis $\{e\}$ and let $P$ denote the Markov operator acting on $B\big(\mathcal{F}(\mathcal{H})\big)$, the algebra of bounded linear operators on the full Fock space $\mathcal{F}(\mathcal{H})$ over $\mathcal{H}$, defined by $P(x) = l^* x l$ where $l \in B\big(\mathcal{F}(\mathcal{H})\big)$ is the left creation operator associated with $e$, that is, $l(x) = e \otimes x$ for $x \in \mathcal{F}(\mathcal{H})$. Then $H^{\infty}\Big(B\big(\mathcal{F}(\mathcal{H})\big), P\Big)$ as a von Neumann algebra is isomorphic to $L^{\infty}(\mathbb{T})$. 
	
	Let us now explain the setting and the main object of study of this paper. Throughout this article, $\mathcal{H}$ denotes a separable Hilbert space (finite or infinite-dimensional) with an orthonormal basis 
	$\{e_i : i \in \Theta\}$ where $\Theta$ stands for the set $\{1, 2, \cdots, n\} (n \in \mathbb{N}, n > 1)$ or the set $\mathbb{N}$ of natural numbers and $\omega = \{\omega_i : i \in \Theta\}$ is a sequence of strictly positive real numbers such that $\sum_{i \in \Theta} \omega_i = 1$. We define a Markov operator $P_{\omega}$ acting on $B\big(\mathcal{F}(\mathcal{H})\big)$ by
	\begin{align*}
		P_{\omega}(x) = \sum_{i \in \Theta} \omega_i l_{e_i}^* x l_{e_i}, \ x \in B(\mathcal{F}(\mathcal{H})),
	\end{align*}
	where $l_{e_i}$ is the left creation operator associated with $e_i$, $i \in \Theta$. The purpose of this article is to study the non-commutative Poisson boundary $H^{\infty}\Big(B\big(\mathcal{F}(\mathcal{H})\big), P_{\omega}\Big)$. We will see later in Section $3$ that the Poisson boundary does not depend on the choice of the orthonormal basis of $\mathcal{H}$. Stated more precisely, if we choose another orthonormal basis $\{f_i : i \in \Theta\}$ of $\mathcal{H}$ and consider the Markov operator $P_{\omega}^\prime$ on $B\big(\mathcal{F}(\mathcal{H}) \big)$ given by $$P_{\omega}^\prime (x) = \sum_{i \in \Theta} \omega_i l_{f_i}^* x l_{f_i}, \ x \in B(\mathcal{F}(\mathcal{H})),$$
	then the von Neumann algebras $H^{\infty}\Big(B\big(\mathcal{F}(\mathcal{H})\big), P_{\omega}\Big)$ and $H^{\infty}\Big(B\big(\mathcal{F}(\mathcal{H})\big), P_{\omega}^\prime\Big)$ are isomorphic.
	
	In what follows, for notational convenience, we will simply use the notation $H^{\infty}$ to denote $H^{\infty}\Big(B\big(\mathcal{F}(\mathcal{H})\big), P_{\omega}\Big)$. In the special case when $\mathcal{H}$ is finite-dimensional with $\dim \mathcal{H} > 1$, and $\omega$ is the constant sequence $\frac{1}{\dim \mathcal{H}}$, we will sometimes use the notation $H^{\infty}_{\dim \mathcal{H}}$ to denote $H^{\infty}$.
	
	We are now in a position to highlight the main results of this paper. Our first important  result is Theorem \ref{masa} that demonstrates a diffuse \textit{masa} (maximal abelian subalgebra) in $H^\infty$ (see Theorem \ref{masa} for more details).
		\begingroup
	\setcounter{tmp}{\value{theorem}}
	\setcounter{theorem}{0} 
	\renewcommand\thetheorem{\Alph{theorem}}
	\begin{theorem}
		Let $\mathcal{D}$ denote the diagonal subalgebra of $B\big(\mathcal{F}(\mathcal{H})\big)$. Then $\mathcal{D} \cap H^\infty$ is a diffuse masa in $H^\infty$.	
	\end{theorem}
	
	Next we summarize our results regarding the type classification of $H^\infty$ (see Remark \ref{Cunz-1}, Theorem \ref{type III} and Corollary \ref{type-rational} for the precise formulations).
	\begin{theorem}
		With notations as above, $H^{\infty}$ is the $\sigma$-weak closure of the Cuntz algebra generated by the right creation operators and hence, is injective. For any choice of the sequence $\omega$, $H^\infty$ is always a factor of type $III$ . Further, if $\mathcal{H}$ is finite-dimensional and if $G$ is the closed subgroup of $\R_+^*$  generated by $\{ \omega_1, \omega_2, \cdots, \omega_{\dim \mathcal{H}} \}$,
		then  	 
		\[
		H^{\infty} \ \text{is} \ 
		\begin{cases}
			\text{type} \ {III}_\lambda,\quad \text{ if }  ~~G = \{ \lambda^n: ~ n \in \Z\}, 0 < \lambda < 1, \text{ and }\\
			\text{type} \ {III}_1, \quad \text{ if }~~ G = \R_+^*.
		\end{cases} 
		\]
		Moreover, in the case when $\mathcal{H}$ is finite-dimensional, if $H^\infty$ is of type $III_\lambda$ for some real number $\lambda \in (0, 1)$, then $\lambda$ is an algebraic number and if, in particular, $\lambda$ is rational, then $\lambda$ must be of the form $\frac{1}{k}$ for some natural number $k$. In particular, if $\mathcal{H}$ is finite-dimensional and $\omega$ is the constant sequence $\frac{1}{\dim \mathcal{H}}$, then $H^\infty$ is a factor of type $III_{\frac{1}{\dim \mathcal{H}}}$.
	\end{theorem}

	Next we address the following question. 
	\begin{question}\label{question}
		Given a unitary operator $U : \mathcal{H} \rightarrow \mathcal{H}$, does there exist an automorphism ($*$-algebra isomorphism) of $H^{\infty}$ that takes $r_{\xi}$ to $r_{U \xi}$ for $\xi \in \mathcal{H} ?$ (Here, for any $\xi \in \mathcal{H}$, $r_\xi$ denotes the right creation operator associated with $\xi$.)	
	\end{question}
	This question and its treatment is inspired in part by the second quantization procedure on free Araki-Woods von Neumann algebras (see \cite{BIK}, \cite{ChEr}) (more generally, on $q$-Araki-Woods \cite{MaWa} or on mixed $q$-Araki-Woods algebras \cite{BKM}) that is an indispensable tool for obtaining approximation properties. 
	We summarize our results below, referring the reader to Theorem \ref{quan} for a precise statement.
	\begin{theorem}
		Let $\mathcal{H}$ be finite-dimensional with $\dim \mathcal{H} > 1$, and let $\omega$ be the constant sequence $\frac{1}{\dim \mathcal{H}}$. 
		For each unitary $U$ on $\mathcal{H}$, there is a unique automorphism $\Psi_U$ of $H^\infty_{\dim \mathcal{H}}$ that takes $r_{\xi}$ to $r_{U \xi}$ for $\xi \in \mathcal{H}$. Further, the correspondence $$\mathcal{U}(\mathcal{H}) \ni U \mapsto \Psi_U \in \text{Aut}(H^{\infty}_{\dim \mathcal{H}})$$ of the unitary group $\mathcal{U}(\mathcal{H})$ of $\mathcal{H}$ to $\text{Aut}(H^{\infty}_{\dim \mathcal{H}})$, the automorphism group of $H^{\infty}_{\dim \mathcal{H}}$, is an injective group homomorphism.
	\end{theorem}
	\endgroup
	Below, we briefly discuss the contents of this article.
	
	The material of the Section $2$ is well known and is meant just to set up the notation to be used in the sequel for the convenience of the	reader. In this section we summarize relevant facts concerning
	full Fock spaces over Hilbert spaces.
	
	The Section $3$ begins by observing that the Poisson boundary does not depend on the choice of the orthonormal basis and then proves a technical result (Proposition \ref{multiplications}) that establishes the multiplication rule in the algebra $H^{\infty}$. Next we show in Proposition \ref{state} that the restriction to $H^{\infty}$ of the vector state on $B\big(\mathcal{F}(\mathcal{H})\big)$ induced by the vacuum vector of $\mathcal{F}(\mathcal{H})$ is indeed a faithful normal state on $H^{\infty}$ and finally, we conclude this section with Theorem \ref{factor} which proves that $H^{\infty}$ is an infinite factor. 
	
	The Section $4$ considers the modular theory for the GNS representation of $H^{\infty}$ associated with the faithful normal state obtained in the preceding section. One of the main contributions of this section is Proposition \ref{generator} which shows that $H^{\infty}$, as a von Neumann algebra, is the $\sigma$-weak closure of the Cuntz algebra $\mathcal{O}_{\dim \mathcal{H}}$ generated by $\{r_{e_i} : i \in \Theta\}$ (where, recall that, for any $\xi \in \mathcal{H}$, $r_{\xi}$ denotes the right creation operator associated with $\xi$). 
	
	The next section is devoted to showing that the abelian von Neumann subalgebra of $H^{\infty}$ which is the intersection of $H^{\infty}$ with the diagonal subalgebra of $B \big(\mathcal{F}(\mathcal{H})\big)$, is a diffuse 
	maximal abelian subalgebra in $H^{\infty}$.
	
	In the penultimate section (Section $6$) we discuss the centralizer of $H^{\infty}$ and its factoriality. The main result of this section is Theorem \ref{type III} which shows that $H^{\infty}$ is a type $III$ factor for any choice of the sequence $\omega$, and if
	$\mathcal{H}$ is finite-dimensional, we completely classify $H^{\infty}$ in terms of its Connes' $S$ invariant.

	The main result of the final section $7$ is Theorem \ref{quan} which proves that, in the case when $\mathcal{H}$ is finite-dimensional ($\dim \mathcal{H} > 1$) and $\omega$ is the constant sequence $\frac{1}{\dim \mathcal{H}}$, the subgroup of  the automorphism group of $H^{\infty}_{\dim \mathcal{H}}$ consisting of all those automorphisms of $H^{\infty}_{\dim \mathcal{H}}$ which preserve the set $\{r_{\xi}: \xi \in \mathcal{H}\}$ is isomorphic to $\mathcal{U}(\mathcal{H})$, the unitary group of $\mathcal{H}$.

	\section{Fock spaces}

	This section is devoted to recalling various standard facts concerning full Fock
	spaces over Hilbert spaces and establishing the terminology and notation that we follow later. The reader may consult \cite{krp92} or \cite{Jns09}, for instance, for proofs and details.
	As mentioned in the introduction, $\mathcal{H}$ denotes a separable Hilbert space with an orthonormal basis $\{e_i : i \in \Theta \}$ where $\Theta$ stands for the set $\{1, 2, \cdots, n\}$ ($n \in \mathbb{N}, n > 1$) or the set $\mathbb{N} = \{1, 2, 3, \cdots\}$.
	We consider the full Fock space over $\mathcal{H}$ defined by
	\begin{align*}
		\mathcal{F}(\mathcal{H}) = \underset{n \geq 0}{\oplus} \mathcal{H}^{\otimes n}
	\end{align*}
	where $\mathcal{H}^{\otimes 0} := \mathbb{C}\Omega$ and for $n \geq 1, \mathcal{H}^{\otimes n}$ is the (Hilbert) tensor product of $n$-copies of $\mathcal{H}$. Here, $\Omega$ is fixed complex number with modulus 1 and we  refer it as vacuum vector. 
	
	For the sake of convenience, we shall introduce some notations.  
	Let $\Lambda$ and $\Lambda^*$ denote respectively the sets 
	\begin{align*}
		\Lambda = \bigcup_{n \geq 0} \Theta^n, \ \mbox{and} \ \Lambda^* = \bigcup_{n \geq 1} \Theta^n, 
	\end{align*}
	where for $n \geq 1, \Theta^n $ denotes the $n$-fold Cartesian product of $\Theta$ and $\Theta^0:= \{()\}$, where $()$ is the empty tuple. The elements of $\Theta^n$, $n \geq 1$, are referred to as sequences of length $n$. If $I$ is a sequence of length $n$, $n \geq 1$, say $I = (i_1, i_2, \cdots, i_n)$, for the interest of notational convenience, we shall, in the sequel, simply write $i_1 i_2 \cdots i_n$ for $I$. We adopt the convention that the empty tuple has length $0$. If $I$ is a sequence in $\Lambda$, we denote its length by $\vert I \rvert$. For $I = ()$, the empty tuple, we set
	\begin{align*}
		e_I:= \Omega,
	\end{align*}
	and for $I \in \Lambda^*$, say $I = i_1 i_2 \cdots i_k$, we define
	\begin{align*}
		e_I:= e_{i_1} \otimes e_{i_2} \otimes \cdots \otimes e_{i_k}.
	\end{align*} 
	Obviously, with the above notation, $\mathcal{B}:= \{e_I : I \in \Lambda\}$ is an orthonormal basis for $\mathcal{F}(\mathcal{H})$. We refer to the elements of $\mathcal{B}$ as simple basis elements.  
	Further, for $I \in \Lambda^*$, we let $I^{op}$ denote the sequence which is the reverse of $I$, that is, if
	\begin{align*}
		I = i_1 i_2 \cdots i_k, \ \mbox{then} \ I^{op} = i_k \cdots i_2 i_1.	
	\end{align*}
	Let $I = i_1 i_2 \cdots i_k \in \Theta^k$, $k \geq 0$. For any non-negative integer $m \leq k$, we denote by $I_m$ the subsequence of $I$ of length $m$ defined by
	\begin{align*}
		I_m = \ \mbox{empty tuple, if} \ m = 0 \ \mbox{and} \ I_m = i_1 i_2 \cdots i_m \ \mbox{if} \ 1 \leq m \leq k.
	\end{align*} 
	If $I, J$ are two sequences in $\Lambda$, then $IJ$ will denote the sequence of length $\lvert I \rvert + \lvert J \rvert$ obtained by juxtaposition. That is, if $I= i_1 i_2 \cdots i_k$, and $J = j_1 j_2 \cdots j_l$, then
	$$IJ = i_1 i_2 \cdots i_k j_1 j_2 \cdots j_l.$$
	Further, given a sequence $I \in \Lambda$ and an integer $n \geq 1$, we use the notation $I^n$ to represent the sequence $\underbrace{I I \cdots I}_\text{$n$-times}$.

	For each $\xi \in \mathcal{H}$, the left creation operator associated with $\xi$, denoted $l_\xi$, is the
	bounded linear operator on $\mathcal{F}(\mathcal{H})$ that satisfies
	\begin{align*}
		l_\xi(\eta) = \xi \otimes \eta, \ \eta \in \mathcal{F}(\mathcal{H}).
	\end{align*}
	Similarly, the right creation operator associated with $\xi$, denoted $r_\xi$, is defined by
	\begin{align*}
		r_\xi(\eta) = \eta \otimes \xi, \ \eta \in \mathcal{F}(\mathcal{H}).
	\end{align*}
	The adjoint $l_\xi^*$ (resp., $r_\xi^*$) of $l_\xi$ (resp., $r_\xi$)  is called the left (resp., right)
	annihilation operator associated with $\xi$. One can easily verify that $l_\xi^*$ and $r_\xi^*$
	satisfy the relations:
	\begin{align*}
		l_\xi^*(\Omega) = 0, \ \mbox{and} \ 
		l_\xi^*(\eta_1 \otimes \eta_2 \otimes \cdots \otimes \eta_k) = \langle \eta_1, \xi \rangle \eta_2 \otimes \cdots \otimes \eta_k \ \mbox{for} \ k \geq 1, \eta_i \in \mathcal{H},
	\end{align*}
	(where $\eta_2 \otimes \cdots \otimes \eta_k = \Omega$ for $k = 1$) and 
	\begin{align*}
		r_\xi^*(\Omega) = 0, \ \mbox{and} \  
		r_\xi^*(\eta_1 \otimes \eta_2 \otimes \cdots \otimes \eta_k) = \langle \eta_k, \xi \rangle \eta_1 \otimes \cdots \otimes \eta_{k-1} \ \mbox{for} \ k \geq 1, \eta_i \in \mathcal{H},
	\end{align*}
	(where $\eta_1 \otimes \cdots \otimes \eta_{k-1} = \Omega$ for $k = 1$). For $i \in \Theta$, we simply use the notation $r_i$ (resp., $l_i$) to denote $r_{e_i}$ (resp., $l_{e_i}$).
	For any $I \in \Lambda$, let $l_I, r_I$ be defined by
	\begin{align*}
		r_I = l_I = 1  \  (\mbox{the identity operator on} \ \mathcal{F}(\mathcal{H})) \ \mbox{if} \ I \  \mbox{is the empty tuple},	
	\end{align*}
	and if $I = i_1 i_2 \cdots i_k$ with $k \geq 1$, then
	\begin{align*}
		r_I = r_{i_1} r_{i_2} \cdots r_{i_k} \ \mbox{and} \ 
		l_I = l_{i_1} l_{i_2} \cdots l_{i_k}.
	\end{align*}	
	We list a few simple facts regarding these operators as a lemma, for the  convenience of reference in the sequel.
	\begin{lemma}\label{form1} With  notations discussed above, we have: 
		\begin{itemize}
			\item[(i)] $l_i^* p_{\Omega} = r_i^* p_{\Omega} = p_{\Omega} l_i = p_{\Omega} r_i = 0$ for all $i \in \Theta$, where $p_{\Omega}$ denotes the orthogonal projection of $\mathcal{F}(\mathcal{H})$ onto $\mathbb{C}\Omega$.
			\item[(ii)] For any $\xi, \eta$ in $\mathcal{H}$, 
			\begin{align*}
				r_\xi ^*r_\eta =l_\xi^* l_\eta =\langle \eta ,\xi \rangle 1,r_\xi^*l_\eta = l_\eta r_\xi^* + \langle \eta, \xi \rangle p_{\Omega} \ \mbox{and} \ l_\eta^*r_\xi = r_\xi l_\eta^* + \langle \xi, \eta \rangle p_{\Omega}.
			\end{align*}
			In particular, 
			\begin{align*}
				r_i^* r_j = l_i^* l_j= \delta_{i, j}, r_i^* l_j = l_j r_i^* + \delta_{i, j} p_{\Omega} \ \mbox{and} \ l_i^* r_j = r_j l_i^* + \delta_{i, j} p_{\Omega} \ \mbox{for all} \ i, j \in \Theta,
			\end{align*}
			where $\delta_{i, j}$ denotes the Kronecker delta function.
		\end{itemize}
		
	\end{lemma}

	\section{$H^{\infty}$: an infinite factor}
	Recall from the introduction that the main object of study in this article is the non-commutative Poisson boundary associated with the pair $\Big(B\big(\mathcal{F}(\mathcal{H})\big), P_{\omega} \Big)$ where $\omega = \{\omega_i : i \in \Theta\}$ is a sequence of positive real numbers such that $\underset{i \in \Theta}{\sum} \omega_i = 1$ and $P_{\omega} : B\big(\mathcal{F}(\mathcal{H})) \rightarrow B\big(\mathcal{F}(\mathcal{H}))$ is the normal UCP map defined by 
	\begin{equation}\label{pform}
		P_{\omega}(x) = \underset{i \in \Theta}{\sum} \omega_i l_i^* x l_i, \ x \in B\big(\mathcal{F}(\mathcal{H})\big),
	\end{equation}	
	where the series on the right converges in the strong operator topology. 
	This section begins with the observation that the Poisson boundary does not depend on the choice of the orthonormal basis and then proves a technical result (Proposition \ref{multiplications}) that establishes the multiplication rule in the algebra $H^{\infty}$. Next we show in Proposition \ref{state} that the restriction to $H^{\infty}$ of the vector state on $B\big(\mathcal{F}(\mathcal{H})\big)$ given by $x \mapsto \langle x \Omega, \Omega \rangle$, $x \in B\big(\mathcal{F}(\mathcal{H}))$,  is indeed a faithful normal state on $H^{\infty}$ and finally, we conclude this section with Theorem \ref{factor} which proves that $H^{\infty}$ is an infinite factor. 

	Before we proceed to prove the results, we pause with a digression concerning notations. We introduce some notations that will be frequently used throughout the article. We set $\omega_{I}=1$ for $I = ()$, the empty tuple, and for $I=i_1\dots i_k \in \Lambda^*$, we set $\omega_{I} = \omega_{i_1}\cdots \omega_{i_k}$. 
	
	With the notations as above, it follows from the definition of $P_{\omega}$ as given by Equation \eqref{pform} that given $x \in B\big(\mathcal{F}(\mathcal{H}))$ and $I, J \in \Lambda$, 
	\begin{align*}
		\big\langle P_{\omega}(x)(e_J), e_I \big\rangle 
		=\big\langle \underset{i \in \Theta}{\sum} \omega_i l_i^* x l_i e_J, e_I \big\rangle
		=  \underset{i \in \Theta}{\sum} \omega_i \big\langle x (e_{iJ}), e_{iI} \big\rangle.
	\end{align*} 
	Consequently,  an element $x \in B\big(\mathcal{F}(\mathcal{H})\big)$ is in $H^{\infty}$ if and only if it satisfies 
	\begin{equation}\label{pform1}
		\langle x(e_J), e_I\rangle = \underset{i \in \Theta}{\sum} \omega_i \big\langle x (e_{iJ}), e_{iI} \big\rangle \ \mbox{for all $I, J \in \Lambda$}.
	\end{equation}

	Since $r_i$ commutes with $l_j$ for all $i, j \in \Theta$, it is evident from the formula for $P_{\omega}$ that 
	$P_{\omega}(r_I)=r_I$ for all $I \in \Lambda$ and hence, $r_I, r_I^*\in H^{\infty}$.  
	Note also that by virtue of Lemma \ref{formula}, $x \circ y$, the Choi-Effros product of two elements $x, y \in H^{\infty}$, is given by
	\begin{equation}\label{formu}
		x \circ y = \ \mbox{SOT-}\lim_{n \rightarrow \infty} P_{\omega}^n(xy).
	\end{equation} 
	We now remark that the Poisson boundary does not depend on the choice of the orthonormal basis. Let $\{f_i : i \in \Theta\}$ be another orthonormal basis of $\mathcal{H}$ and consider the Markov operator $P_{\omega}^{\prime}$ on $B\big(\mathcal{F}(\mathcal{H})\big)$ defined by $$P_{\omega}^\prime (x) = \sum_{i \in \Theta} \omega_i l_{f_i}^* x l_{f_i}, \ x \in B \big(\mathcal{F}(\mathcal{H})\big).$$
	Our goal is to show that $H^\infty\Big(B\big(\mathcal{F}(\mathcal{H})\big), P_{\omega} \Big)$ and $H^\infty\Big(B\big(\mathcal{F}(\mathcal{H})\big), P_{\omega}^\prime \Big)$ are isomorphic as von Neumann algebras. 
	
	Given a unitary $U : \mathcal{H} \rightarrow \mathcal{H}$, recall that the corresponding second quantization $\Gamma_U : \mathcal{F}(\mathcal{H}) \rightarrow \mathcal{F}(\mathcal{H})$ is the unitary operator on $\mathcal{F}(\mathcal{H})$ defined by $\Gamma_U(\Omega) = \Omega$ and $\Gamma_U|_{\mathcal{H}^{\otimes n}} = U^{\otimes n}$ for $n \geq 1$.
	Clearly, $\Gamma_U$ induces the automorphism $\widetilde{\Gamma}_U$ of $B \big(\mathcal{F}(\mathcal{H})\big)$ given by $\widetilde{\Gamma}_U(x) = \Gamma_U x \Gamma_U^*$ for $x \in B \big(\mathcal{F}(\mathcal{H})\big)$. One can easily verify that
	\begin{equation}\label{eqquan}
		\widetilde{\Gamma}_U(p_{\Omega}) = p_{\Omega}, \widetilde{\Gamma}_U(l_{\xi}) = l_{U \xi}, \ \mbox{and} \ \widetilde{\Gamma}_U(r_{\xi}) = r_{U \xi} \ \mbox{where} \ \xi \in \mathcal{H}.
	\end{equation} 
	Consider the unitary $V$ on $\mathcal{H}$ that takes $e_i$ to $f_i$ for $i \in \Theta$.
It is not hard to see that $$ H^\infty\Big(B\big(\mathcal{F}(\mathcal{H})\big), P_{\omega} \Big) \ni x \mapsto \widetilde{\Gamma}_V(x) \in H^\infty\Big(B\big(\mathcal{F}(\mathcal{H})\big), P_{\omega}^\prime \Big)$$ is a $*$-algebra isomorphism. 

	Our next proposition 
	computes several multiplication formulae in $H^{\infty}$ which will be frequently used in the sequel.
	\begin{proposition}\label{multiplications}
		Let $x \in H^{\infty}$ and let $I, J \in \Lambda^*$. Then:
		\begin{itemize}
			\item[(i)] $x \circ r_I = x r_I$.
			\item[(ii)] $r_I^* \circ x = r_I^* x$.
			\item[(iii)] $r_J^*\circ  x \circ r_I = r_J^* x r_I$.
			\item[(iv)] $r_I \circ x = r_I x + \overset{\lvert I \rvert}{\underset{t=1}{\sum}} \omega_{(I^{op})_t} r_{I_{\lvert I \rvert -t}}p_{\Omega} x l_{(I^{op})_t}$.
			In particular, $r_i\circ x=r_ix+\omega_i p_{\Omega}xl_i$, $i \in \Theta$.
			\item[(v)] 
			$x\circ r_I^* = xr_I^*+ \overset{\lvert I \rvert}{\underset{t=1}{\sum}} \omega_{(I^{op})_t}l_{(I^{op})_t}^* x p_{\Omega} r_{I_{\lvert I \rvert -t}}^*$.	
			In particular, $x\circ r_i^*=xr_i^*+\omega_i l_i^*x p_{\Omega}$, $i \in \Theta$.
			\item[(vi)] $x\circ r_I \circ  r_{J}^* = x r_{I} r_{J}^* + \overset{\lvert J \rvert}{\underset{t=1}{\sum}} \omega_{(J^{op})_t}l_{J^{op}_t}^* x r_I p_{\Omega} r_{J_{\lvert J \rvert-t}}^*$.
			In particular, for any $i, j \in \Theta$, $x\circ r_i\circ  r_j^*=xr_ir_j^*+\omega_jl_j^*xr_i p_{\Omega}$.
			\item[(vii)] $r_I \circ r_J^* \circ x = r_I r_J^* x + \overset{\lvert I \rvert}{\underset{t=1}{\sum}} \omega_{(I^{op})_t} r_{I_{\lvert I \rvert-t}} p_{\Omega} r_J^* x l_{(I^{op})_t}$.
			In particular, for any $i, j \in \Theta$, $r_i\circ r_j^*\circ x=r_ir_j^*x+\omega_ip_{\Omega}r_j^*xl_i $.
		\end{itemize}	
	\end{proposition}
	\begin{proof}
		\begin{itemize}
			\item[(i)] A simple computation, using the fact that $r_I$ commutes with $l_i$ for $i \in \Theta$, shows that 
			\begin{align*}
				P_{\omega}(x r_I) = \underset{i \in \Theta}{\sum} \omega_i l_i^* x r_I l_i = \big(\underset{i \in \Theta}{\sum} \omega_i l_i^* x l_i\big) r_I =  P_{\omega}(x) r_I = x r_I
			\end{align*}
			and hence, $xr_I \in H^{\infty}$. It now follows by an appeal to Equation \eqref{formu} that $x \circ r_I = xr_I$. 
			\item[(ii)] Follows from part (i) of the proposition by taking adjoint.
			\item[(iii)] Clearly, $r_J^* \circ x \circ r_I = r_J^* \circ (xr_I) = r_J^*xr_I$ where the first equality follows from part(i) of the proposition whereas the second equality follows from part (ii) of the proposition.			
			\item[(iv)] We prove the result by induction on $\lvert I \rvert$, the length of the sequence $I$. Let $\lvert I \rvert = 1$, say, $I = i$ for some $i \in \Theta$. Then 
			\begin{align*}
				P_{\omega}(r_i x) &= \sum_{j \in \Theta} \omega_j l_j^* (r_i x) l_j
				= \sum_{j \in \Theta} \omega_j (l_j^* r_i) x l_j \\
				&= \sum_{j \in \Theta} \omega_j (r_i l_j^* + \delta_{i, j} p_{\Omega}) x l_j \  \ \mbox{\big(by Lemma \ref{form1}(ii)\big)}\\	
				&= r_i \Big(\sum_{j \in \Theta} \omega_j l_j^* x l_j \Big) + \omega_i p_{\Omega}x l_i\\
				&= r_i P_{\omega}(x) + \omega_i p_{\Omega}x l_i = r_i x + \omega_i p_{\Omega}x l_i, 
			\end{align*}
			and hence,
			\begin{align*}
				P^2_{\omega} (r_i x) &= P_{\omega}(r_i x) + \omega_i P_{\omega}\big(p_{\Omega}x l_i\big)\\&= P_{\omega}(r_i x) + \omega_i \sum_{j \in \Theta} \omega_j l_j^* \big(p_{\Omega} x l_i\big) l_j\\
				& = P_{\omega}(r_i x) \  \ \big(\mbox{since} \ l_j^* p_{\Omega} = 0 \ \mbox{by Lemma \ref{form1}(i)}\big). 
			\end{align*}
			Consequently, $P_{\omega}^n(r_ix) = P_{\omega}(r_ix)$ for all $n \geq 1$ and so, an appeal to Equation \eqref{formu} immediately shows that
			\begin{equation}\label{form2}
				r_i \circ x = P_{\omega}(r_i x) = r_i x + \omega_i p_{\Omega}x l_i.
			\end{equation}
			Thus the result is true for all $r_I$ with $|I|=1$.
			Suppose that the result is true for all sequences of length $m$ for some $m \geq 1$. We show that the result is true for all sequences of length $m+1$. Let $I$ be a sequence of length $m+1$, say, $I = i_1 i_2 \cdots i_{m+1}$. An appeal to Proposition \ref{multiplications}(i) yields that $r_I = r_{i_1} \cdots r_{i_{m+1}} =
			r_{i_1} \circ \cdots \circ r_{i_{m+1}}$. Thus,
			\begin{align*}
				r_I \circ x = r_{i_1} \circ (r_{i_2} \circ \cdots r_{i_{m+1}} \circ x ) = r_{i_1} \circ (r_J \circ x)
			\end{align*}
			where $J = i_2 \cdots i_{m+1}$. Since $\lvert J \rvert = m$, by induction hypothesis we have 
			\begin{align*}
				r_J \circ x = r_Jx + \sum_{t = 1}^{m} \omega_{J^{op}_t} r_{J_{m-t}} p_{\Omega} x l_{J^{op}_t}, 
			\end{align*}
			and hence, an application of Equation \eqref{form2} shows that
			\begin{align*}
				r_I \circ x &= r_{i_1} \circ (r_J \circ x) \\
				&= r_{i_1}(r_J \circ x) + \omega_{i_1} p_{\Omega}(r_J \circ x)l_{i_1} \ \big(\mbox{by an appeal to Equation \eqref{form2}}\big)\\
				&= r_{i_1}\big(r_J x + \sum_{t = 1}^{m} \omega_{J^{op}_t} r_{J_{m-t}} p_{\Omega} x l_{J^{op}_t}\big) + \omega_{i_1} p_{\Omega}\big(r_J x + \sum_{t = 1}^{m} \omega_{J^{op}_t} r_{J_{m-t}} p_{\Omega} x l_{J^{op}_t}\big) l_{i_1}	\\
				& = r_I x + \sum_{t = 1}^{m} \omega_{J^{op}_t} r_{i_1} r_{J_{m-t}} p_{\Omega} x l_{J^{op}_t} + \omega_{i_1} \omega_{J^{op}} p_{\Omega} x l_{J^{op}} l_{i_1}	\\
				&\big(\mbox{since, by Lemma \ref{form1}}, \ p_{\Omega} r_i = 0 \ \mbox{for} \ i \in \Theta, \ \mbox{so} \ p_{\Omega} r_{J_{m-t}} = 0 \ \mbox{for} \ 0 \leq t < m\big)\\
				&= r_I x + \sum_{t = 1}^{m} \omega_{I^{op}_t} r_{I_{m+1-t}} p_{\Omega} x l_{I^{op}_t} + \omega_{I^{op}} p_{\Omega} x l_{I^{op}} \\
				&= r_I x + \sum_{t = 1}^{m+1} \omega_{I^{op}_t} r_{I_{m+1-t}} p_{\Omega} x l_{I^{op}_t}.
			\end{align*}
			Thus the result is true for all sequences of length $m+1$. Hence, by the priciple of mathematical induction, the result is true for all sequences in $\Lambda^*$.
			\item[(v)] Follows from part (iv) of the proposition by taking adjoint.
			\item[(vi)] By part (i) of the proposition, we have $x \circ r_I \circ r_J^* = (x r_I) \circ r_J^*$ and then an application of part (v) of the proposition yields the desired result.
			\item[(vii)] Follows from part (vi) of the proposition by taking adjoint. 
		\end{itemize}
	\end{proof}
	Now consider the vector state on $B\big(\mathcal{F}(\mathcal{H})\big) $ induced by the vacuum vector $\Omega$ , i.e, consider the following state
	\begin{align*}
		x\mapsto \langle x\Omega ,\Omega \rangle, \  \ x \in B\big(\mathcal{F}(\mathcal{H})\big).  
	\end{align*}
	This is a normal state on $B\big(\mathcal{F}(\mathcal{H})\big)$.
	Let $\varphi$ denote its restriction to $H^{\infty}$, that is,
	\begin{align*}
		\varphi : H^{\infty} \ni x \mapsto \langle x\Omega ,\Omega \rangle \in \mathbb{C}.
	\end{align*} 
	We assert that $\varphi$ is indeed \textit{a faithful, normal state on the von Neumann algebra}  $H^{\infty}$. The following lemma plays a crucial role towards establishing our assertion.
	\begin{lemma}{\label{lemma 1}} 
		Let $x$ be a positive element of $B\big(\mathcal{F}(\mathcal{H})\big)$ such that $x \in H^{\infty}$. Then $\langle x\Omega, \Omega\rangle =0 $ implies $x = 0$.
	\end{lemma}
	\begin{proof}
		Since $x$ is a positive element of $B\big(\mathcal{F}(\mathcal{H})\big)$ such that $x \in H^{\infty}$ and $\langle x\Omega, \Omega\rangle =0$, in order to prove that $x = 0$, it suffices to show that $\langle x e_I, e_I \rangle = 0$ for all $I \in \Lambda^*$. Since $x \in H^{\infty}$, we have that
		\begin{align*}
			0 = \langle x\Omega, \Omega\rangle
			= \big\langle \big(\sum_{i \in \Theta} \omega_i l_i^* x l_i \big) \Omega, \Omega \big\rangle
			= \sum_{i \in \Theta} \omega_i \langle l_i^* x l_i \Omega, \Omega \rangle = \sum_{i \in \Theta} \omega_i \langle x e_i, e_i \rangle.
		\end{align*}
		As $x$ is positive element of $B\big(\mathcal{F}(\mathcal{H})\big)$, $\langle x e_i, e_i \rangle \geq 0$ for all $i \in \Theta$ and since $\omega_i > 0$ for each $i \in \Theta$, it follows that 
		\begin{align*}
			\langle x e_i, e_i \rangle = 0 \ \mbox{for all} \ i \in \Theta, \ \mbox{or, equivalently,}  \ \langle r_i^* x r_i \Omega, \Omega \rangle = 0, \ \mbox{for each} \ i \in \Theta.
		\end{align*} 
		As $x$ is a positive element of $B\big(\mathcal{F}(\mathcal{H})\big)$, so is $r_i^*xr_i$ for any $i \in \Theta$. Since
		$\langle r_i^* x r_i \Omega, \Omega \rangle = 0$ and $r_i^* x r_i  \in H^\infty$, the similar argument as before shows that
		\begin{align*}
			\langle r_j^* r_i^* x r_i r_j \Omega, \Omega \rangle = 0, \ \mbox{that is,} \ \langle r_{ij}^* x r_{ij} \Omega, \Omega\rangle = 0 \ \mbox{for all} \ i, j \in \Theta.
		\end{align*} 
		Continuing this way we conclude that
		\begin{align*}
			\langle r_I^* x r_I \Omega, \Omega \rangle = 0 \ \mbox{for all} \ I \in \Lambda^*, \ \mbox{or, equivalently,} \ \langle x e_I, e_I \rangle = 0, \ \mbox{for all} \ I \in \Lambda^*.
		\end{align*}
		This completes the proof.
	\end{proof}
	Let $x$ be a positive element of $H^{\infty}$. As $H^{\infty}$ is a $C^*$-algebra, so, $x = y^* \circ y$ for some $y \in H^{\infty}$. It follows from the product rule as given by Equation \eqref{formu} that
	\begin{align*}
		x = \text{SOT-} \lim_{n \rightarrow \infty} P_\omega^n(y^* y)
	\end{align*}
	and since $P_\omega^n(y^*y)$ is positive element of $B\big(\mathcal{F}(\mathcal{H})\big)$ for all $n \geq 0$, we see that $x$, being the strong limit of a sequence of positive elements of $B\big(\mathcal{F}(\mathcal{H})\big)$, is positive in $B\big(\mathcal{F}(\mathcal{H})\big)$. Thus we conclude that a positive element of $H^{\infty}$ is also a positive element of $B\big(\mathcal{F}(\mathcal{H})\big)$. As an immediate consequence of this observation and Lemma \ref{lemma 1} we obtain that:
	\begin{proposition}\label{state}
		With notations as above, the linear functional $\varphi$ is a faithful, normal state on $H^{\infty}$.
	\end{proposition}
	\begin{proof}
		Since the map 
		\begin{align*}
			x \mapsto \langle x \Omega, \Omega \rangle, \ x \in B\big(\mathcal{F}(\mathcal{H})\big),
		\end{align*}
		is a normal state on $B\big(\mathcal{F}(\mathcal{H})\big)$ and $H^{\infty}$ is a $\sigma$-weak closed operator system in $B\big(\mathcal{F}(\mathcal{H})\big)$, $\varphi$ is a normal state on $H^{\infty}$. To prove that $\varphi$ is faithful, we note that given a positive element $x$ of $H^{\infty}$ such that $\varphi(x) = 0$, by virtue of the discussion preceding this proposition, it follows that $x$ is indeed a positive element of $B\big(\mathcal{F}(\mathcal{H})\big)$ and then an appeal to Lemma \ref{lemma 1} immediately shows that $x = 0$, completing the proof.
	\end{proof}
	The next result is an easy consequence of Proposition \ref{multiplications}.
	\begin{coro}\label{phi}
		Let $x\in H^{\infty}$ and $I,J\in \Lambda$. Then:
		\begin{itemize}\addtolength{\itemsep}{0.4cm}
			\item[(i)] $\varphi (x\circ r_{J}^*)=\omega_{J} \langle x\Omega, r_{J}\Omega \rangle =\omega_{J} \varphi(r_J^*\circ x)$.
			\item[(ii)] $\varphi ( r_{J}\circ x)=\omega_{J}\langle xr_{J}\Omega, \Omega \rangle =\omega_{J}\varphi(x\circ r_J)$.
			\item[(iii)] $\varphi (x\circ r_{I}\circ r_{J}^*)=\omega_{J} \langle xr_{I}\Omega, r_{J}\Omega \rangle =\omega_{J}\varphi (r_J^*\circ x\circ r_I)$.
			\item[(iv)] $\varphi ( r_{I} \circ r_{J}^*\circ  x)=\omega_{I} \langle xr_{I}\Omega, r_{J}\Omega \rangle =\omega_{I}\varphi (r_J^*\circ x\circ r_I)$.
		\end{itemize}
	\end{coro}
	\begin{proof}
		\begin{itemize}
			\item[(i)] If $J = ()$, the empty tuple, then there is nothing to prove. Let $J \in \Theta^k$ for some $k \geq 1$, say, $J = j_{i_1} j_{i_2} \cdots j_{i_k}$. 
			Since $r_i^* (\Omega) = 0$ for any $i \in \Theta$, it follows from the formula for $x \circ r_J^*$ as given by Proposition \ref{multiplications}(v) that 
			\begin{align*}
				(x \circ r_J^*) \Omega &= (x r_J^*) \Omega + \sum_{t = 1}^k \omega_{J^{op}_t} \big(l_{J^{op}_t}^* x p_{\Omega} r_{I_{k-t}}^*\big) \Omega
				= \omega_J \big(l_{J^{op}}^* x \big) \Omega
			\end{align*}
			and hence,
			\begin{align*}
				\varphi(x \circ r_J^*) &= \langle (x \circ r_J^*) \Omega, \Omega \rangle 
				= \omega_{J} \langle \big(l_{J^{op}}^* x\big) \Omega, \Omega \rangle 
				= \omega_J \langle x \Omega, l_{J^{op}} \Omega \rangle. 
			\end{align*}
			Since $l_{J^{op}} \Omega = r_J \Omega$, it follows from the preceding equation that
			\begin{align*}
				\varphi(x \circ r_J^*) = \omega_J \langle x \Omega, r_J \Omega \rangle 
				= \omega_J \langle (r_J^*x) \Omega, \Omega \rangle 
				= \omega_J \langle \big(r_J^* \circ x\big) \Omega, \Omega \rangle 
			\end{align*}
			where the last equality follows by an appeal to Proposition \ref{multiplications}(ii).
			\item[(ii)] Follows from part (i) by taking adjoints.
			\item[(iii)] Follows from part (i) of the corollary and Proposition \ref{multiplications}(i).
			\item[(iv)] Follows from part (iii) of the corollary by taking adjoints.
		\end{itemize}
	\end{proof}
	The following simple lemma will be useful in the proof of Theorem \ref{factor}. 
	\begin{lemma}{\label{lemma 2}}
		Let $x$ be an element of the center of 
		$H^{\infty}$. Then $\langle xr_{J}\Omega,\Omega \rangle =0$ for any $J \in \Lambda^*$. Further, if $I, J \in \Lambda$ are of the same length, then $r_{I}^*xr_{J}=\delta_{I,J}  x$. 
	\end{lemma}
	\begin{proof}
		Note that for any $J \in \Lambda^*$,
		\begin{align*}
			\langle x r_J \Omega, \Omega \rangle 
			= \varphi (x\circ r_J)
			=\varphi(r_J \circ x)	= \omega_{J} \langle xr_{J}\Omega,\Omega \rangle, 		
		\end{align*}
		where the first equality follows from Proposition \ref{multiplications}(i), the second equality is a consequence of the fact that $x$ lies in the center of $H^{\infty}$ and the last equality follows from Corollary \ref{phi}(ii).
		Clearly, $\omega_J \neq 1$ as $\lvert J \rvert \geq 1$ and hence, it follows from the preceding equation that $\langle x r_J \Omega, \Omega \rangle = 0$.
		Further, if $I, J \in \Lambda$ are of the same length, then
		\begin{align*}
			r_I^* x r_J &=r_{I}^*\circ x\circ r_{J} \  \ \mbox{\big(by Proposition \ref{multiplications}(iii)\big)}\\
			&=x\circ r_{I}^*\circ r_{J} \ \ \big(\mbox{since} \ x \ \mbox{is in the center of} \  \ H^{\infty}\big)\\
			&= x\circ r_{I}^*r_{J} \  \ \big(\mbox{by Proposition \ref{multiplications}(ii)}\big)\\
			&= \delta _{I,J}x.
		\end{align*} 
	\end{proof}
	We are now ready to prove the main result of this section.
	\begin{theorem}{\label{factor}}
		$H^{\infty}$ is an infinite factor.
	\end{theorem}
	\begin{proof}
		To prove that $H^{\infty}$ is a factor, it suffices to show that given any $x$ in the center of $H^{\infty}$, $\langle x e_I, e_J \rangle = \delta_{I,J} \langle x \Omega, \Omega \rangle$ for all $I, J \in \Lambda$. Let us take an element $x$ in the center of $H^{\infty}$ and let $I, J \in \Lambda$, say, $I = i_1 i_2 \cdots i_k$ and $J = j_1 j_2 \cdots j_m$. First consider the case when $I, J$ are sequences of different lengths. Without loss of generality we may assume that $\lvert I \rvert > \lvert J \rvert$. Then $I^{op} = i_k \cdots i_2 i_1$. Set $I^{\prime} = i_k \cdots i_{k-m+1}$ and $I^{\prime \prime} = i_{k-m} \cdots i_2 i_1$ so that $I^{op} = I^{\prime} I^{\prime \prime}$. Since $k > m$, $\lvert I^{\prime \prime} \rvert \geq 1$ and hence, by virtue of Lemma \ref{lemma 2}, we obtain that $\langle x r_{I^{\prime \prime}} \Omega, \Omega \rangle = 0$. Consequently,
		\begin{align*}
			\langle x e_I, e_J \rangle &= \big\langle \big(xr_{I^{op}} \big) \Omega, r_{J^{op}} \Omega \big\rangle \\
			&= \big\langle \big(r_{J^{op}}^* x r_{I^{op}}\big) \Omega, \Omega \big\rangle \\
			&= \big\langle \big(r_{J^{op}}^* x r_{I^{\prime}} r_{I^{\prime \prime}} \big) \Omega, \Omega \big\rangle \ \  \big(\mbox{as $I^{op} = I^{\prime} I^{\prime \prime}$, so $r_{I^{op}} = r_{I^{\prime}} r_{I^{\prime \prime}}$} \big) \\
			&= \delta_{J^{op}, I^{\prime}}\langle xr_{I^{\prime \prime}}\Omega,\Omega \rangle \  \ \big(\mbox{ as $\lvert I^{\prime} \rvert = \lvert J \rvert = m$, by Lemma \ref{lemma 2}, $r_{J^{op}}^* x r_{I^{\prime}} =  \delta_{J^{op}, I^{\prime}} x$} \big)\\
			&=0.
		\end{align*}
		If $I, J \in \Lambda$ are of the same length, then an appeal to Lemma \ref{lemma 2} shows that
		\begin{align*}
			\langle x e_I, e_J \rangle = \big\langle \big(r_{J^{op}}^* x r_{I^{op}} \big) \Omega, \Omega \big\rangle  = \delta_{I^{op}, J^{op}} \langle x \Omega, \Omega \rangle = \delta_{I, J} \langle x \Omega, \Omega \rangle.
		\end{align*}
		Thus we have proved that $\langle x e_I, e_J \rangle = \delta_{I, J} \langle x \Omega, \Omega \rangle$ for all $I, J \in \Lambda$ and hence, $H^{\infty}$ is a factor.
		This is an infinite factor because $r_i^*\circ r_i=1$ but $r_i\circ r_i^*\neq 1$ for every $i \in \Lambda$.		
	\end{proof}

	\section{Modular Theory}
	Recall from Proposition \ref{state} that the functional
	$\varphi : H^{\infty} \ni x \mapsto \langle x \Omega, \Omega \rangle$ is a faithful normal state on $H^{\infty}$.
	Let $(\mathcal{H}_{\varphi},\pi_{\varphi},\Omega_{\varphi} )$ denote the GNS triple associated with the state $\varphi$ where $\mathcal{H}_{\varphi}$ is a Hilbert space, $\pi_{\varphi} :H^{\infty} \rightarrow B(\mathcal{H}_{\varphi})$ is the normal isometric $*$-homomorphism, and $\Omega_{\varphi} \in \mathcal{H}_{\varphi}$ is the cyclic and separating vector for $\pi_{\varphi}(H^{\infty})$ such that
	\begin{align*}
		\varphi(x) = \langle \pi_{\varphi}(x) \Omega_{\varphi}, \Omega_{\varphi} \rangle_{\varphi},  \ x \in H^{\infty},
	\end{align*} 
	where $\langle,\rangle_{\varphi}$ denotes the inner product of $\mathcal{H}_{\varphi}$.
	Let $S_0$ denote the densely defined closable 
	conjugate-linear operator, with domain $\pi_{\varphi}(H^{\infty})\Omega_{\varphi}$, defined by
	\begin{align*}
		S_0 \big(\pi_{\varphi}(x) \Omega_{\varphi}\big) = \pi_{\varphi}(x^*) \Omega_{\varphi},  \ x \in H^{\infty}.
	\end{align*}
	Let $S$ denote the closure of $S_0$ and let $F$ denote $S^*$. Let $S=J_{\varphi}\Delta_{\varphi} ^{\frac{1}{2}}$ be the polar decomposition of $S$. The operators $J_{\varphi}$ and $\Delta_{\varphi}$ are called, respectively, the modular conjugation and the modular operator associated with the pair $(H^{\infty}, \varphi)$. For each $t\in \mathbb{R}$, we denote by $\sigma_t^{\varphi}$ the $*$-automorphism of $H^{\infty}$ defined by $\sigma_t^{\varphi}(x):=\pi_{\varphi}^{-1}\left(\Delta_{\varphi} ^{it}\pi_{\varphi}(x)\Delta_{\varphi}^{-it}\right)$ for $x\in H^{\infty}$. The one-parameter group $\{\sigma_t^{\varphi}: t \in \mathbb{R}\}$ of $*$-automorphisms of $H^{\infty}$ is called the 
	group of modular automorphisms of $H^{\infty}$ associated with $\varphi$.
	It is a fact (see \cite[Lemma $3^0$, Page 279]{StZs}) that $\pi_{\varphi}(H^{\infty})\Omega_{\varphi}\subseteq \text{Dom}(\Delta_{\varphi})$ (domain of $\Delta_{\varphi}$).
	
	\begin{proposition}{\label{delta}}
		
		\begin{itemize}
			\item[(i)] For $I,J\in \Lambda$, $\Delta_{\varphi} \left(\pi_{\varphi}(r_{I}\circ r_{J}^*)\Omega_{\varphi}\right)=\frac{\omega_{I}}{\omega_{J}} \pi_{\varphi}(r_{I}\circ r_{J}^*)\Omega_{\varphi}$. 
			\item[(ii)] For $I,J\in \Lambda$, $J_{\varphi}\left(\pi_{\varphi}(r_{I}\circ r_{J}^*)\Omega_{\varphi}\right)=\sqrt{\frac{\omega_{J}}{\omega_{I}}}\pi_{\varphi}(r_J \circ r_I^*)\Omega_{\varphi}$.
			\item[(iii)] For all $I,J\in \Lambda$ and $t\in \mathbb{R}$, 
			$\sigma_t^{\varphi}(r_{I}\circ r_{J}^*)=(\frac{\omega_{I}}{\omega_{J}}) ^{\emph{i}t}r_{I}\circ r_{J}^*$.		
		\end{itemize}	
	\end{proposition}
	\begin{proof}
		\begin{itemize}
			\item[(i)] 
			Note that for any $x \in H^{\infty}$,
			\begin{align*}
				\big\langle \Delta_{\varphi} \big(\pi_{\varphi}(r_I \circ r_J^*) \Omega_{\varphi}\big), \pi_{\varphi}(x) \Omega_{\varphi}\big\rangle_{\varphi} 
				&= \big\langle S^*S \big(\pi_{\varphi}(r_I \circ r_J^*) \Omega_{\varphi}\big),\pi_{\varphi}(x) \Omega_{\varphi}\big\rangle_{\varphi} \\
				&=\langle  \pi_{\varphi}(x^*) \Omega_{\varphi},\pi_{\varphi}(r_J \circ r_I^*)\Omega_{\varphi}\rangle_{\varphi} \  \ (\mbox{since $S$ is anti-linear})\\
				&=\varphi (r_{I}\circ r_{J}^*\circ x^*)\\
				&=\omega_{I} \varphi (r_{J}^* \circ x^* \circ r_I) \  \   (\mbox{by Corollary \ref{phi}(iv)})\\
				& = \frac{\omega_{I}}{\omega_{J}} \varphi(x^* \circ r_I \circ r_J^*) \  \ (\mbox{by Corollary \ref{phi}(iii)})\\
				&= \frac{\omega_{I}}{\omega_{J}}\big\langle \pi_{\varphi}(r_{I}\circ r_{J}^*)\Omega_{\varphi},\pi_{\varphi}(x) \Omega_{\varphi}\big\rangle_{\varphi}.
			\end{align*}
			This completes the proof.	
			\item[(ii)]
			It follows from part (i) of the proposition that
			\begin{align*}
				\Delta_{\varphi}^{\frac{1}{2}} \big(\pi_{\varphi}(r_{I}\circ r_{J}^*)\Omega_{\varphi}\big)=\sqrt{\frac{\omega_{I}}{\omega_{J}}} \pi_{\varphi}(r_{I}\circ r_{J}^*) \Omega_{\varphi}
			\end{align*} 
			and hence,
			\begin{align*}
				\Delta_{\varphi}^{-\frac{1}{2}} (\pi_{\varphi}(r_I \circ r_J^*) \Omega_{\varphi})=\sqrt{\frac{\omega_J}{\omega_I}} \pi_{\varphi}(r_{I}\circ r_{J}^*) \Omega_{\varphi}.
			\end{align*}
			Hence, the relation $S=J_{\varphi}\triangle_{\varphi} ^{\frac{1}{2}}$ yields that
			\begin{align*}
				J_{\varphi}\big(\pi_{\varphi}(r_{I}\circ r_{J}^*)\Omega_{\varphi}\big)=\sqrt{\frac{\omega_{J}}{\omega_{I}}}\pi_{\varphi}(r_J \circ r_I^*)\Omega_{\varphi}.
			\end{align*}
			\item[(iii)] It follows from the part (i) of the proposition that for $I, J \in \Lambda$, $$\Delta_{\varphi}^{\emph{i}t} \big(\pi_{\varphi} (r_I \circ r_J^*) \Omega_{\varphi} \big)= \Big(\frac{\omega_I}{\omega_{J}}\Big)^{\emph{i}t}\pi_{\varphi} (r_I \circ r_J^*) \Omega_{\varphi}$$ and consequently,
			$$\big(\Delta_{\varphi}^{\emph{i}t} \pi_{\varphi}(r_I \circ r_J^*) \Delta_{\varphi}^{-\emph{i}t} \big)\Omega_{\varphi} = \Big(\frac{\omega_I}{\omega_{J}}\Big)^{\emph{i}t}\pi_{\varphi} (r_I \circ r_J^*) \Omega_{\varphi}.$$
			Since by Tomita's theorem, $\Delta_{\varphi}^{it} \pi_{\varphi}(H^{\infty}) \Delta_{\varphi}^{-it} =\pi_{\varphi}( H^{\infty})$ for all $t \in \mathbb{R}$ and
			since $\Omega_{\varphi}$ is a separating vector for $\pi_{\varphi}(H^{\infty})$, it follows from the equation above that $\Delta_{\varphi}^{\emph{i}t} \pi_{\varphi}(r_I \circ r_J^*) \Delta_{\varphi}^{-\emph{i}t} = \Big(\frac{\omega_I}{\omega_{J}}\Big)^{\emph{i}t} \pi_{\varphi}(r_I \circ r_J^*)$. Thus, $\sigma_t^{\varphi}(r_I \circ r_J^*) = \Big(\frac{\omega_I}{\omega_{J}}\Big)^{\emph{i}t} r_I \circ r_J^*$. 
		\end{itemize}
	\end{proof}
In the next proposition  we show that 
$H^{\infty}$, as a von Neumann algebra, is generated by $\{r_i : i \in \Theta\}$.
	\begin{proposition}{\label{generator}}
		$H^{\infty}$, as a von Neumann algebra, is generated by $\{r_i : i \in \Theta\}$.
	\end{proposition}
	\begin{proof}
		Let $\mathcal{N}$ denote the von Neumann subalgebra of $H^{\infty}$ generated by $\{r_i: i \in \Theta\}$, that is, $\mathcal{N}$ is the $\sigma$-weak closure of the (unital self-adjoint) subalgebra of $H^{\infty}$ consisting of polynomials in $r_i, r_i^*$ for all $i \in \Theta$. An appeal to Proposition \ref{delta}(iii) together with normality of $\sigma_t^{\varphi}$ yields that $\sigma_t^{\varphi}(\mathcal{N})\subseteq \mathcal{N}$ for all $t \in \mathbb{R}$. 
		Hence, by \cite[Proposition 2.6.6]{Sun}, there exists a $\varphi$-compatible conditional expectation $E$ of $H^{\infty}$ onto $\mathcal{N}$. To prove the proposition, it suffices to show that $E(x)=x$ for all $x \in H^{\infty}$. Given $x \in H^{\infty}$, to show that $x - E(x) = 0$, it suffices to verify that $\big\langle \big(x - E(x)\big) e_I, e_J \big\rangle = 0$ for all $I, J \in \Lambda$, or, equivalently, $\big\langle \big(x - E(x)\big) r_I \Omega, r_J \Omega \big\rangle = 0$ for all $I, J \in \Lambda$.
		Note that 
		\begin{align*}
			\omega_J \big\langle \big(x - E(x)\big) r_I \Omega, r_J \Omega \big\rangle 
			&= \varphi \Big(\big(x - E(x)\big) \circ r_I \circ r_J^* \Big) \  \ (\mbox{by Corollary \ref{phi}(iii)})\\
			&=\varphi (x \circ r_I \circ r_J^*) - \varphi \big(E(x) \circ r_I \circ r_J^* \big)\\
			&= \varphi (x \circ r_I \circ r_J^*) - \varphi \big(E(x \circ r_I \circ r_J^*)\big) \  \ (\mbox{as $E$ is $\mathcal{N}$-$\mathcal{N}$-bilinear})\\
			&= \varphi (x \circ r_I \circ r_J^*) - \varphi (x \circ r_I \circ r_J^*) \  \ (\mbox{by $\varphi$-compatibility of $E$})\\
			&= 0.
		\end{align*}
		This completes the proof.
	\end{proof}

	As an immediate consequence of the preceding proposition we obtain that:
	\begin{coro}\label{total}
		The set $\mathcal{S} := \{ \pi_{\varphi}(r_I \circ r_J^*)\Omega_{\varphi} : I, J \in \Lambda\}$ is total in $\mathcal{H}_{\varphi}$.
	\end{coro}
	\begin{proof}
		It follows from Proposition \ref{generator} that $\text{span} \{ \pi_{\varphi}(r_I \circ r_J^*) : I, J \in \Lambda\}$
		is strongly dense in $\pi_{\varphi}(H^{\infty})$. This observation together with the fact that $\Omega_{\varphi}$ is cyclic for $\pi_{\varphi}(H^{\infty})$ yields the desired result.	
	\end{proof}
	We conclude this section with the following result which computes the spectrum of $\Delta_{\varphi}$.
	\begin{lemma}\label{spec}
		The spectrum of $\Delta_{\varphi}$ equals the closure of $\{\frac{\omega_{I}}{\omega_{J}}: I, J \in \Lambda\}$.	
	\end{lemma} 
	\begin{proof}
		The proof of the lemma follows from Proposition \ref{delta}(i) and Corollary \ref{total}.
	\end{proof}

\begin{remark}\label{Cunz-1}
Let $\mathcal{O}_{\dim \mathcal{H}}$ denote the Cuntz algebra generated by $\{r_i : i \in \Theta\}$. Then from Proposition \ref{generator}, we can realize the Poisson boundary  $H^\infty$ as  $\mathcal{O}_{\dim \mathcal{H}}^{\prime \prime} $. As a consequence of the this fact, we  observe  the following;  
\begin{enumerate}
	\item Since $\mathcal{O}_{\dim \mathcal{H}}$ is nuclear, so $H^\infty$  is injective. \\
	\item 
	In the case when ${\dim \mathcal{H}} < \infty$, the GNS  representation of $\mathcal{O}_{\dim \mathcal{H}}$ with respect to a KMS state has been studied in \cite{Iz4} and also, the type of $\mathcal{O}_{\dim \mathcal{H}}^{\prime \prime}$ has been determined. Using those results we may obtain some identical  results regarding the type classification of   $H^\infty$ and we address  this issue in Section 6 in details.

\end{enumerate}

\end{remark}

		\section{A Diffuse Masa in $H^{\infty}$}
	Consider the diagonal subalgebra $\mathcal{D}$ of $B\big(\mathcal{F}(\mathcal{H})\big)$, that is, $\mathcal{D}$ consists of all those elements $x \in B\big(\mathcal{F}(\mathcal{H})\big)$ such that $x$ is a diagonal operator with respect to the orthonormal basis $\mathcal{B}$ of $\mathcal{F}\big(\mathcal{H}\big)$, that is, $\varphi(r_I^* \circ x \circ r_J) = 0$ for all $I, J \in \Lambda$ with $I \neq J$. Needless to say, $\mathcal{D}$ is a \textit{masa} (maximal abelian subalgebra) in $B\big(\mathcal{F}(\mathcal{H})\big)$. 
	Let $\mathcal{D}_{\omega} := \mathcal{D} \cap H^{\infty}$. It is evident that $\mathcal{D}_{\omega}$ is an abelian von Neumann subalgebra of $H^{\infty}$. This section is devoted to showing that $\mathcal{D}_{\omega}$ is diffuse and 
	is a \textit{masa} in $H^{\infty}$. 
	
	Recall that an abelian von Neumann subalgebra $A$ of a von Neumann algebra $M$ is a masa in $M$ if and only if $A^{\prime} \cap M = A$. A subalgebra $A$ of a von Neumann algebra $M$ is said to be diffuse if it has no minimal projection.
	
	We first show that:
	\begin{proposition}\label{diaggen}
		$\mathcal{D}_{\omega}$, as a von Neumann algebra, is generated by $\{r_I \circ r_I^* : I \in \Lambda\}$.
	\end{proposition}
	\begin{proof}
		Given $I \in \Lambda$, a little thought should convince the reader that for any $K, L\in \Lambda$ with $K \neq L$, 
		$$\varphi(r_L^* \circ r_I \circ r_I^* \circ r_K) = \frac{\omega_K}{\omega_L} \varphi(r_I^* \circ r_K \circ r_L^* \circ r_I) = 0,$$
		showing that $r_I \circ r_I^* \in \mathcal{D}_\omega$.
		The proof for the remaining part of the proposition is similar to that of Proposition \ref{generator} and is left to the reader.	  
	\end{proof}

	Next we prove that:
	\begin{proposition}\label{diffuse}
		$\mathcal{D}_\omega$ is diffuse in $H^{\infty}$. 
	\end{proposition}
	\begin{proof}
		If possible let there be a non-zero minimal projection $q$ in $\mathcal{D}_{\omega}$. Note that as $q$ is a positive element of $B\big(\mathcal{F}(\mathcal{H})\big)$, $\varphi(r_I^* \circ q \circ r_I) \geq 0$ for all $I \in \Lambda$ and as $q$ is non-zero, we must have that $\varphi(q) > 0$. If for every positive integer $n$, there is exactly one sequence $I_n \in \Theta^n$ (of length $n$) such that $\varphi(r_{I_n}^* \circ q \circ r_{I_n}) > 0$, then one can easily see that $$\varphi(r_{I_n}^* \circ q \circ r_{I_n}) = \frac{1}{\omega_{I_n}} \varphi(q)$$ which is an impossibility as the sequence $\{ \frac{1}{\omega_{I_n}}\}$ diverges to infinity. Let $m$ be the smallest positive integer such that there are at least two distinct sequences $I, J$ in $\Theta^m$ with $\varphi(r_I^* \circ q \circ r_I) > 0$ and $\varphi(r_J^* \circ q \circ r_J) > 0$. Clearly, $q \circ r_I \circ r_I^* = \alpha q$ for some $\alpha \in \mathbb{C}$ and hence, $r_I^* \circ q \circ r_I = r_I^* \circ (q \circ r_I \circ r_I^*) \circ r_I = r_I^* \circ (\alpha q) \circ r_I = \alpha (r_I^* \circ q \circ r_I)$ and since $r_I^* \circ q \circ r_I$ is non-zero element, it follows that $\alpha = 1$. On the other hand, it follows from $q \circ r_I \circ r_I^* = \alpha q$ that
		$$\alpha (r_J^* \circ q \circ r_J) = r_J^* \circ q \circ r_I \circ r_I^* \circ r_J = 0,$$ and as $r_J^* \circ q \circ r_J$ is non-zero element, it follows that $\alpha = 0$  and thus, we arrive at a contradiction. This completes the proof.  
	\end{proof}
	We are now ready to prove the main result of this section.
	\begin{theorem}\label{masa}
		$\mathcal{D}_{\omega}$ is a diffuse masa in $H^{\infty}$.
	\end{theorem}
	\begin{proof}
		We have already proved in Proposition \ref{diffuse} that $\mathcal{D}_{\omega}$ is diffuse in $H^\infty$.
		Thus, in order to complete the proof, it just remains to prove that $\mathcal{D}_{\omega}$ is a masa in $H^{\infty}$. 
		Let $x \in (\mathcal{D}_{\omega})^{\prime} \cap H^{\infty}$. We need to show that $x \in \mathcal{D}_{\omega}$, or, equivalently, $\varphi(r_I^* \circ x \circ r_J) = 0$ for all $I, J \in \Lambda$ with $I \neq J$. 
		Observe that by virtue of Equation \eqref{pform1} it suffices to see that $\varphi(r_I^* \circ x \circ r_J) = 0$ for all $I, J \in \Lambda^*$ with $I \neq J$. 
		Consider $I, J \in \Lambda^*$ with $I \neq J$.
		\vspace{2mm} \\ 
		\textbf{Case 1.} Assume first that neither $I$ is of the form $JK$ nor $J$ is of the form $IK$ for some $K \in \Lambda$ and so, in this case, $r_I^* \circ r_J = 0$. As $x$ commutes with every element of $\mathcal{D}_{\omega}$, it follows that $$x \circ (r_I \circ r_I^* - r_J \circ r_J^*) = (r_I \circ r_I^* - r_J \circ r_J^*) \circ x,$$
		and consequently,
		$$r_I^* \circ x \circ (r_I \circ r_I^* - r_J \circ r_J^*) \circ r_J = r_I^* \circ  (r_I \circ r_I^* - r_J \circ r_J^*) \circ x \circ r_J.$$ 
		As $r_I^* \circ r_J = 0$, one can easily see that the expression on the left of the equality sign in the preceding equation equals $- r_I^* \circ x \circ r_J$ whereas the expression on the right of the equality sign in the preceding equation equals $r_I^* \circ x \circ r_J$ and thus, we have that $r_I^* \circ x \circ r_J = 0$ and therefore, $\varphi(r_I^* \circ x \circ r_J) = 0$.
		\vspace{2mm} \\ 
		\textbf{Case 2.} Now assume that either $I$ is of the form $JK$ or $J$ is of the form $IK$ for some $K \in \Lambda^*$. Without loss of generality we may assume that $J = IK$ for some $K \in \Lambda^*$. As $x \circ r_J \circ r_J^* = r_J \circ r_J^* \circ x$, we have that $r_I^* \circ (x \circ r_J \circ r_J^*) \circ r_J = r_I^* \circ (r_J \circ r_J^* \circ x) \circ r_J$ from which it follows that $r_I^* \circ x \circ r_J = r_I^* \circ r_J \circ r_J^* \circ x \circ r_J = r_K \circ r_J^* \circ x \circ r_J$ and consequently,
		$$\varphi(r_I^* \circ x \circ r_J) = \varphi(r_K \circ r_J^* \circ x \circ r_J) = \omega_K \varphi(r_J^* \circ x \circ r_J \circ r_K) = \omega_K \varphi(r_J^* \circ x \circ r_{JK}),$$
		that is,
		$$\varphi(r_I^* \circ x \circ r_{IK}) = \omega_K \varphi(r_{IK}^* \circ x \circ r_{IKK}).$$
		Repeated application of this shows that 
		$$\varphi(r_I^* \circ x \circ r_{IK}) = \omega_K \varphi(r_{IK}^* \circ x \circ r_{IKK}) = \omega_K^2 \varphi(r_{IKK}^* \circ x \circ r_{IKKK}) = \ldots = \omega_K^n \varphi(r_{IK^n}^* \circ x \circ r_{IK^{n+1}})$$ for any $n \geq 1$ and hence, $\varphi(r_{IK^n}^* \circ x \circ r_{IK^{n+1}}) = \omega_K^{-n} \varphi(r_I^* \circ x \circ r_{IK})$. Since $0 < \omega_K < 1$, $\omega_K^{-n} \uparrow \infty$ as $n \to \infty$, and this forces that $\varphi(r_I^* \circ x \circ r_{IK}) = 0$. This completes the proof.
	\end{proof}

	\section{Centraliser and type decomposition}
		In this section we discuss the centralizer of $H^{\infty}$ and its factoriality. The main result of this section is Theorem \ref{type III} which shows that $H^{\infty}$ is a type $III$ factor for any choice of the sequence $\omega$. If
		$\mathcal{H}$ is finite-dimensional, then we completely classify  $H^{\infty}$ in terms of its Connes' $S$ invariant.

		Recall that (see \cite[Section 10.27]{StZs}, \cite[Definition 2.5.13]{Sun}) the centralizer of a faithful normal state  $\theta$ on $H^{\infty}$, denoted $H^{\infty}_{\theta}$, is the von Neumann subalgebra of $H^{\infty}$ defined by
		\begin{align*}
			H^{\infty}_{\theta} = \{x \in H^{\infty}: \sigma_t^{\theta}(x) = x, \  \text{for all} \  t \in \mathbb{R}\}.
		\end{align*}
		It is a fact (see \cite[Section 10.27]{StZs}, \cite[Corollary 2.5.14] {Sun}) that 
		\begin{align*}
			H^{\infty}_{\theta} = \{x \in H^{\infty} : \theta(x \circ y) = \theta(y \circ x) \ \mbox{for all} \ y \in H^{\infty}\}.
		\end{align*}	
		Also, let the center of the centralizer be denoted by $\mathcal{Z}(H^{\infty}_{\theta})$.	\\\\
		Recall that (Proposition \ref{state}) the vacuum state  $\varphi(\cdot)=\langle (\cdot)\Omega, \Omega \rangle$ on  $H^{\infty}$, is a faithful normal state on $H^{\infty}$. It is immediate from Proposition \ref{delta}(iii) that $r_I \circ r_I^* \in H^\infty_\varphi$ for all $I \in \Lambda$ and hence, it follows by an appeal to Proposition \ref{diaggen} that $\mathcal{D}_\omega \subseteq H^\infty_\varphi$ and consequently, $\mathcal{Z}(H^{\infty}_{\varphi}) \subseteq \mathcal{D}_\omega$. 

		Let us briefly recall Connes' classification of type $III$ factors. Connes in \cite{Connes-C} defined the $S$ invariant of a factor $M$, denoted $S(M)$, to be the intersection
		over all faithful normal and semifinite weights $\theta$ of the spectra of the modular operators $\Delta_\theta$, that is, 
		$$S(M)   = \cap\{ \text{ spectrum of } \Delta_\theta:~~ \theta \text{ is a faithful normal semifinite weight on } M \}.$$
		$M$ is a type $III$ factor if an only if $0 \in S(M)$; in that case, Connes' $III_\lambda$ classification, $\lambda \in [0, 1]$, in terms of its $S$ invariant is as follows: 
		\[
		S(M)=
		\begin{cases}
			\{ \lambda^n: ~ n \in \Z\}\cup \{0\},\quad\text{iff } M \text{ is type } {III}_\lambda, \quad \lambda \in (0, 1)\\ 
			[0, \infty ), ~ \text{ iff } M \text{ is type } {III}_1\\
			\{0, 1\} ~ \text{ iff } M \text{ is type } {III}_0.
		\end{cases} 
		\]
		Let $\Gamma(M) = \R^*_+ \cap S(M)$ where $\R^*_+$ is the multiplicative group of positive real numbers.  Then $\Gamma(M)$ is a closed subgroup of $\R^*_+$. It is known that a non-trivial closed subgroup of $\R^*_+$ is cyclic, i.e, of the form $\{ \lambda^n: ~ n \in \Z\}$ for some $ 0< \lambda <1$. Thus, a type $III$ factor $M$ is of type
		\begin{itemize}
			\item[(i)] $III_0$ if $\Gamma(M) = \{1\}$;
			\item[(ii)] $III_\lambda$ if $\Gamma(M) = \{\lambda^n : n \in \mathbb{Z}\}$ (for $\lambda \in (0, 1)$);
			\item[(iii)] $III_1$ if $\Gamma(M) = (0, \infty)$.
		\end{itemize}
	The remaining of this section is dedicated to showing  that $H^\infty$ is a factor of type $III$ for any choice of the sequence $\omega$ and then completely classify $H^\infty$, in the case when $\mathcal{H}$ is finite-dimensional, in terms of its Connes' $S$ invariant. It is well-known that if for some faithful normal state $\theta$ on $H^\infty$, the centralizer $H^\infty_\theta$ turns out to be a factor, then $\Gamma(H^\infty_\theta) = \mathbb{R}^*_+ \cap (\text{ spectrum of } \Delta_\theta)$ (see, for instance, \cite[Proposition 3.4.7]{Sun}). We next show that in the case when $\mathcal{H}$ is finite-dimensional, $H^\infty_\varphi$ is indeed a factor and then appeal to the aforementioned result to classify $H^\infty$ in this case. 

	Assume that $\mathcal{H}$ is finite-dimensional, say, $\dim \mathcal{H} = n > 1$. We proceed to prove that $H^\infty_\varphi$ is a factor. The strategy of the proof is similar to that of \cite[Lemma 4.4]{Iz1993}. But we use the information of the centralizer and thus our proof becomes straightforward.  
	
	Consider the following endomorphism $\alpha$ on $H^\infty$ defined by $$\alpha(x) = \sum_{i = 1}^n r_i \circ x \circ r_i^*,~~ x \in H^\infty.$$ Consider $\mathcal{O}_n$, the Cuntz algebra generated by $\{r_i : 1 \leq i \leq n\}$. Then note that $\alpha$ can be regarded as an endomorphism of $\mathcal{O}_n$.
	
	Now consider the unitary $v = \sum_{i, j = 1}^n r_i \circ r_j \circ r_i^* \circ r_j^*$ and for $k \geq 1$, define $u_k = v \circ \alpha(v) \circ \alpha^2(v) \circ \cdots \circ \alpha^{k-1}(v)$. Observe that $u_k$ is a unitary in $H^\infty_\varphi$ for all $k \in \mathbb{N}$. We recall the following result from \cite[$\S 2$]{DR}.
	\begin{lemma}\cite[$\S 2$]{DR}\label{DR}
	Assume ${\dim \mathcal{H}} < \infty$,  let $I \in \Lambda$ and consider $R = r_I \circ r_I^*$. Then $$\alpha(R) = \lvert \lvert . \rvert\rvert- \lim _{n \to \infty}u_n \circ R \circ u_n^*.$$
	\end{lemma}
We are now ready to prove that $H^\infty_\varphi$ is a factor
\begin{proposition}\label{phifactor}
	With notations as above, $H^\infty_\varphi$ is a factor.	
\end{proposition}
\begin{proof}
	Let $x \in \mathcal{Z}(H^\infty_\varphi)$. It follows, by virtue of Proposition \ref{diaggen}, that there is a net $(x_i)$ in $\text{span}\{r_I \circ r_I^* : I \in \Lambda\}$ such that $\pi_{\varphi}(x_i) \to \pi_{\varphi}(x)$ in SOT. Let $y \in \pi_{\varphi}(H^\infty)'$, then note that \begin{align*}
		\pi_{\varphi}\big( \alpha(x) - x\big) y \Omega_{\varphi} = & \pi_{\varphi}\big( \alpha(x) - \alpha(x_i)\big) y \Omega_{\varphi} + \pi_{\varphi}\big( \alpha(x_i) - u_k \circ x_i \circ u_k^*\big) y\Omega_{\varphi}\\
		& + \pi_{\varphi}( u_k \circ x_i \circ u_k^* - x) y \Omega_{\varphi},
	\end{align*}
and hence,
\begin{align*}
	\norm{\pi_{\varphi}\big( \alpha(x) - x\big) y\Omega_{\varphi}} \leq  & \norm{\pi_{\varphi}\big( \alpha(x) - \alpha(x_i)\big) y\Omega_{\varphi}} + \norm{\pi_{\varphi}\big( \alpha(x_i) - u_k \circ x_i \circ u_k^*\big) y\Omega_{\varphi}}\\
	& + \norm{\pi_{\varphi}( u_k \circ x_i \circ u_k^* - x) y\Omega_{\varphi}}.
\end{align*}
As $\pi_{\varphi}(x_i) \to \pi_{\varphi}(x)$ in SOT, it follows that $\pi_{\varphi}\big(\alpha(x_i)\big) \to \pi_{\varphi}\big(\alpha(x)\big)$ in SOT and hence, $$\norm{\pi_{\varphi}\big( \alpha(x) - \alpha(x_i)\big) y\Omega_{\varphi}} \to 0.$$
Now we obtain the following estimate of the term \\ $\norm{\pi_{\varphi}( u_k \circ x_i \circ u_k^* - x) y\Omega_{\varphi}}$:
\begin{align*}
	\norm{\pi_{\varphi}( u_k \circ x_i \circ u_k^* - x) y\Omega_{\varphi}} &= \norm{ y\pi_{\varphi} \big( u_k \circ (x_i - x)\big) \pi_{\varphi}(u_k^* ) \Omega_{\varphi}}\\
	&= \norm{y\pi_{\varphi}\big( u_k \circ (x_i - x) \big)  J_\varphi \pi_{\varphi}(u_k) J_\varphi\Omega_{\varphi}}\\
	&= \norm{y J_\varphi \pi_{\varphi}(u_k) J_\varphi \pi_{\varphi}\big( u_k \circ (x_i - x) \big) \Omega_{\varphi}}\\
	&\leq \norm{y} \norm{\pi_{\varphi}(x_i - x) \Omega_{\varphi}},
	\end{align*}
and as $\pi_{\varphi}(x_i) \to \pi_{\varphi}(x)$ in SOT, it follows that $\underset{i}{\lim}~ \pi_{\varphi}( u_k \circ x_i \circ u_k^* - x) y\Omega_{\varphi} = 0$.
Thus, given $\epsilon > 0$, there is an index $i_0$ such that $$\norm{\pi_{\varphi}\big( \alpha(x) - \alpha(x_{i_0})\big) y\Omega_{\varphi}} < \frac{\epsilon}{3}, \text{ and } \norm{\pi_{\varphi}( u_k \circ x_{i_0} \circ u_k^* - x) y\Omega_{\varphi}} < \frac{\epsilon}{3}.$$
Finally, an appeal to Lemma \ref{DR} immediately yields that $$\norm{\pi_{\varphi}\big( \alpha(x_{i_0}) - u_k \circ x_{i_0} \circ u_k^*\big) y\Omega_{\varphi}} \leq \norm{y} \norm{\alpha(x_{i_0}) - u_k \circ x_{i_0} \circ u_k^*} \to 0 \text{ as } k \to \infty.$$ 
Thus, there is a positive integer $k_0$ such that  
$$\norm{\pi_{\varphi}\big( \alpha(x_{i_0}) - u_{k_0} \circ x_{i_0} \circ u_{k_0}^*\big) y\Omega_{\varphi}} < \frac{\epsilon}{3}.$$
Consequently, 
\begin{align*}
	\norm{\pi_{\varphi}\big( \alpha(x) - x\big) y\Omega_{\varphi}} \leq  & \norm{\pi_{\varphi}\big( \alpha(x) - \alpha(x_{i_0})\big) y\Omega_{\varphi}} + \norm{\pi_{\varphi}\big( \alpha(x_{i_0}) - u_{k_0} \circ x_{i_0} \circ u_{k_0}^*\big) y\Omega_{\varphi}}\\
	& + \norm{\pi_{\varphi}( u_{k_0} \circ x_{i_0} \circ u_{k_0}^* - x) y\Omega_{\varphi}} \\
	& < \frac{\epsilon}{3} + \frac{\epsilon}{3} + \frac{\epsilon}{3} = \epsilon.
\end{align*}
As $\epsilon > 0$ is arbitrary, it follows that 
$$ \pi_{\varphi}\big( \alpha(x) - x\big) y\Omega_{\varphi} = 0.$$
Since the space $ \pi_{\varphi}(H^\infty)'\Omega_{\varphi}$ is dense in $\mathcal{H}_\varphi$, it follows that $\alpha(x) = x$. This implies that $r_i^* \circ x \circ r_i = x$ for all $i = 1, 2, \cdots, n$ and from this one can easily deduce that $x$ is scalar multiple of the identity element. This completes the proof. 
	
\end{proof}

		We are now ready to state and prove the main result of this section.
		\begin{theorem}{\label{type III}}
			With notations as above, $H^{\infty}$ is a factor of type $III$. Further, if $\mathcal{H}$ is finite-dimensional and if $G$ is the closed subgroup of $\R_+^*$  generated by $\{ \omega_1, \omega_2, \cdots, \omega_{\dim \mathcal{H}} \}$,
			then  	 
			\[
			H^{\infty} \ \text{is} \ 
			\begin{cases}
				\text{type} \ {III}_\lambda,\quad \text{ iff }  ~~G = \{ \lambda^n: ~ n \in \Z\}, 0 < \lambda < 1, \text{ and }\\
				\text{type} \ {III}_1, \quad \text{ iff }~~ G = \R_+^*.
			\end{cases} 
			\]
			In particular, if $\mathcal{H}$ is finite-dimensional and $\omega$ is the constant sequence $\frac{1}{\dim \mathcal{H}}$, then $H^\infty$ is a factor of type $III_{\frac{1}{\dim \mathcal{H}}}$.
		\end{theorem}

		\begin{proof}	
			If possible let $H^{\infty}$ be a semifinite factor. Then it follows from \cite[Theorem 10.29]{StZs} that there exists a group $\{u_t\}_{t \in \mathbb{R}}$ of unitary operators in $H^{\infty}$ such that $\sigma_t^{\varphi}(x)=u_t \circ x \circ u_t^*$ for all $x\in H^{\infty}$ and $t \in \mathbb{R}$. Clearly, $u_t\in  \mathcal{Z}(H^{\infty}_{\varphi})$ for all $t\in \mathbb{R}$.
			For all $I \in \Lambda$ and $t \in \mathbb{R}$, we have $u_t \circ r_I \circ u_t^* = \sigma_t^{\varphi}(r_I) = \omega_{I}^{it} r_I$ and hence, $$r_I^* \circ u_t \circ r_I = r_I^* \circ u_t \circ r_I  \circ u_t^* \circ u_t = \omega_{I}^{it}    r_I^*\circ r_I\circ  u_t  = \omega_{I}^{it} u_t.$$ In particular, for any $j \in \Theta$, $r_j^* \circ u_t \circ r_j = r_j^* u_t r_j = \omega_j^{it} u_t$ for all $t \in \mathbb{R}$. Since, for any $t \in \mathbb{R}$, $u_t = P_{\omega}(u_t) = \sum_{j \in \Theta} \omega_j r_j^* u_t r_j$, the normality of $\varphi$ yields that
			\begin{equation}\label{type}
				\varphi(u_t) = \sum_{j \in \Theta} \omega_j \varphi(r_j^* u_t r_j) = \Big(\sum_{j \in \Theta} \omega_j^{it + 1}\Big) \varphi(u_t
				).
			\end{equation}
			We assert that there is a non-zero $t \in \mathbb{R}$ such that
			$\sum_{j \in \Theta} \omega_j^{it + 1} \neq 1$. To see this fix $k \in \Theta$ and set $t = \frac{1}{\log\omega_k}$. Then clearly $t \neq 0$ as $0 < \omega_k <1$. If $\sum_{j \in \Theta} \omega_j^{it + 1} = 1$, then $\sum_{j \in \Theta} \omega_j^{it + 1} = \sum_{j \in \Theta} \omega_j$ would imply that $\sum_{j \in \Theta} \omega_j(1 - \omega_j^{it}) = 0$. Hence, $$0 = \text{Re} \big(\sum_{j \in \Theta} \omega_j(1 - \omega_j^{it})\big) = \sum_{j \in \Theta} \omega_j \big(1-\cos(t \log \omega_j)\big) = \sum_{j \in \Theta} \omega_j \Big(1-\cos \Big(\frac{\log \omega_j}{\log\omega_k}\Big)\Big).$$ As $1-\cos\big(\frac{\log \omega_j}{\log\omega_k}\big) \geq 0$ for all $j$ and for $j = k$, $1-\cos\big(\frac{\log \omega_k}{\log\omega_k}\big) = 1 - \cos(1) > 0$, it must happen that Re$\big(\sum_{j \in \Theta} \omega_j(1 - \omega_j^{it})\big) > 0$, a contradiction and thus our assetion is established. Choose a non-zero real number $t_0$ such that $\sum_{j \in \Theta} \omega_j^{it_0 + 1} \neq 1$. It then follows from Equation \eqref{type} that $\varphi(u_{t_0}) = 0$. Consequently, for any $I \in \Lambda$, $$\langle u_{t_0} r_I \Omega, r_I \Omega \rangle = \varphi(r_I^* u_{t_0} r_I) = \omega_{I}^{i t_0}\varphi(u_{t_0})=0.$$ 
		As $u_{t_{0}} \in \mathcal{Z}(H^\infty_\varphi) \subset \mathcal{D}_\omega$, we obtain that $u_{t_0} = 0$, leading to a contradiction.
			
			For the later part,	let $\mathcal{H}$ be finite-dimensional, say, $\dim \mathcal{H} = n > 1$.
			 We have already proved in Proposition \ref{phifactor} that $H^\infty_\varphi$ is a factor and hence, an appeal to \cite[Proposition 3.4.7(b)]{Sun} immediately yields that   $S(H^\infty) = \text{  spectrum of } \Delta_\varphi$ which, by virtue of Proposition \ref{delta}(i), equals the closure of $\big\{ \frac{\omega_{I}}{\omega_{J}}: I,J\in \Lambda \big\}$.  The desired result now follows at once from the discussion on Connes' $S$ invariant for a type $III$ factor at the beginning of this section. 
		\end{proof}

		\begin{remark}
		
		In the case when ${\dim \mathcal{H}} < \infty$,	 we recall from \cite{Iz4} the study of the GNS  representation of $\mathcal{O}_{\dim \mathcal{H}}$ with respect to a KMS state $\varphi^{\widetilde{\omega}}$where $ {\widetilde{\omega}} = ( \widetilde{\omega}_1, \widetilde{\omega}_2, \cdots, \widetilde{\omega}_{\dim \mathcal{H}} )$ is a $n$-tuple of positive numbers and $\beta$ is the positive number determined by $\sum_{i=1}^{\dim \mathcal{H}} e^{- \beta  \widetilde{\omega}_i} = 1$. 
		For detailed description of  $\varphi^{\widetilde{\omega}}$, we refer to  \cite{Iz4} and \cite{DEvans}. The corresponding modular automorphism group is given by $ \sigma^{\varphi^{\widetilde{\omega}}}_t(r_j) = e^{-i\beta t {\widetilde{\omega}}_j} r_j$ for $ j = 1, 2, \cdots,  {\dim \mathcal{H}}$. Izumi obtained the following classification results (see \cite[Theorem 4.7]{Iz4}):
		\begin{itemize}
			\item[(i)] If $ \frac{\widetilde{\omega_i}}{\widetilde{\omega_j}} \notin \mathbb{Q}$ for some $i , j$, then $\mathcal{O}_{\dim \mathcal{H}}^{\prime \prime}$ is a type $III_1$ factor.
			\item[(ii)] If $ \frac{\widetilde{\omega_i}}{\widetilde{\omega_j}}\in \mathbb{Q}$ for all $i , j$, then $\mathcal{O}_{\dim \mathcal{H}}^{\prime \prime}$ is a type $III_\lambda $ factor for some $\lambda \in (0, 1)$. \\
		\end{itemize}
		
		On the other hand we proved that (see Proposition \ref{state}) the vacuum state $\varphi(\cdot) = \langle(\cdot) \Omega, \Omega \rangle $ is  faithful and normal on $H^\infty$.  The modular automorphisms  associated to $\varphi$ is given by $ \sigma^{\varphi}_t(r_j) = \omega_j^{it} r_j$ for $t \in \mathbb{R}$ and $ j = 1, 2, \cdots,  {\dim \mathcal{H}}$. Thus, by setting  $ {\widetilde{\omega}}_i = -\frac{1}{\beta} \log(\omega_i)$ for $i = 1, 2, \cdots, \dim \mathcal{H}$, we obtain, by an appeal to the aforementioned theorem of Izumi \cite[Theorem 4.7]{Iz4}, the following classification for $H^\infty$: 
			\begin{itemize}
			\item[(i)] If $ \frac{\log(\omega_i)}{\log(\omega_j)} \notin \mathbb{Q}$ for some $i , j$, then $H^\infty$ is a type $III_1$ factor.
			\item[(ii)] If $\frac{\log(\omega_i)}{\log(\omega_j)} \in \mathbb{Q}$ for all $i , j$, then $H^\infty$ is a type $III_\lambda $ factor for some $\lambda \in (0, 1)$. 
		\end{itemize}
	
	\vskip2mm
		Note that from the perspective of the theory of  non-commutative Poisson boundary, it is natural to consider the vacuum state $\varphi$  and study the corresponding GNS representation and find the relationship between the
		types of $H^\infty$ and the weight $\omega = \{ \omega_1, \omega_2, \cdots, \omega_{\dim \mathcal{H}}\}$ associated with the Makrkov operator $P_\omega$. We also point out that our study of the GNS representation of
		the Cuntz algebra $\mathcal{O}_{\dim \mathcal{H}}$ (generated by $\{r_i
			: 1 \leq i \leq \dim \mathcal{H}\}$) with respect to the vacuum state
		$\varphi$ and subsequently, Connes’ classification of the von Neumann algebra $\mathcal{O}_{\dim \mathcal{H}}^{\prime \prime}$ follows a fairly
		standard path and thus seems quite natural. Further, in our context we also provide additional informations of the type of $H^\infty$
		immediately after this remark .
		
	\end{remark}

		We conclude this section by showing that in the case when $\mathcal{H}$ is finite-dimensional, if $H^{\infty}$ is of type $III_{\lambda}$ for some rational $\lambda \in (0,1)$, then $\lambda$ must belong to the set $\{\frac{1}{k}:k\in \N\}$. For the rest of the section, $\mathcal{H}$ denotes a finite-dimensional Hilbert space, say, $\dim \mathcal{H}=n>1$ and $\omega=\{\omega_1, \cdots ,\omega_n\}$ so that $\Theta=\{1,2,\cdots ,n\}$. In order to prove the result, we need the following lemma.

		\begin{lemma} \label{rational}
			Fix $0<\lambda<1$.  Then the following are equivalent: 
			\begin{enumerate}
				\item There exist $\{ k_i:  i \in \Theta\} \subseteq \N$ with 	$gcd\{k_i:i\in \Theta\}=1$ and $       \overset{}{\underset{i\in \Theta}\sum}  \lambda^{k_i}=1$.  
				\item There exist $\{ c_i \in (0, 1) : ~~ i \in  \Theta \text{ and } \overset{}{\underset{i\in \Theta}\sum} c_i =  1 \}$ such that $ \{ \lambda^k:k\in \Z\} = \{ \frac{c_I}{c_J}: I,J\in \Lambda\}$ where $c_{()}=1$ and $c_I=c_{i_1}\cdots c_{i_m}$ for $I=i_1\cdots i_m\in \Lambda^*$.
			\end{enumerate}
		\end{lemma}
		\begin{proof}
			$(1) \Rightarrow (2)$: 	 For  $i\in \Theta$, set $c_i:=\lambda^{k_i}$. Since $gcd\{k_i:i\in \Theta\}=1$, there exist integers  $\{n_i:i\in\Theta\}$ such that $\overset{}{\underset{i\in \Theta}\sum} n_ik_i=1$.
			Hence,
			\begin{align*}
				\lambda =& \lambda^{ \big(\sum_{ i \in \Theta}  n_ik_i \big)}
				=\prod_{ i \in \Theta} \big(\lambda^{k_i}\big)^{n_i}
				=\prod_{ i \in \Theta} {c_i}^{n_i}.
			\end{align*}
			Consequently, $\lambda\in \{\frac{c_I}{c_J}: I,J\in \Lambda \}$ and thus, it follows that  $\{ \frac{c_I}{c_J}: I,J\in \Lambda\}=\{ \lambda^k:k\in \Z\}$.\\ 
			$(2) \Rightarrow (1)$: 
			Since for $i\in \Theta$, $c_i\in \{ \frac{c_I}{c_J}: I,J\in \Lambda\}=\{\lambda^k:k\in \Z\}$, we have $c_i=\lambda^{k_i}$ for some $k_i\in \Z$. As $0<c_i<1$, $k_i\in \N$. Note that  if $gcd\{k_i:i\in \Theta\}>1$, then $\lambda\notin \{ \frac{c_I}{c_J}: I,J\in \Lambda\}$, leading to a contradiction. Thus, $gcd\{k_i:i\in \Theta\}=1$
		\end{proof}
		As an immediate consequence of the preceding lemma we obtain that:
		\begin{coro}\label{type-rational}
			If $H^{\infty}$ is of type $III_{\lambda}$ for some real number $\lambda\in (0,1)$, then $\lambda$ is algebraic. Moreover, if $  \lambda$ is rational, then $\lambda=\frac{1}{k}$ for some natural number $k>1$. 
		\end{coro}
		\begin{proof}
			If $H^{\infty}$ is of type $III_{\lambda}$ for some real $\lambda\in 
			(0,1)$, it follows from Theorem \ref{type III} that, $\{\frac{\omega_{I}}{\omega_J}:I,J\in \Lambda\}=\{ \lambda^k:k\in \Z\}$ and then an appeal to Lemma \ref{rational} shows that there exist positive integers $\{k_i:i\in \Theta\}$ satisfying $\sum_{i\in \Theta}^{}\lambda^{k_i}=1$  and $gcd\{k_i:i\in \Theta\}=1$. In particular, $\lambda$ is algebraic. If $\lambda$ is rational, say, $\lambda=\frac{p}{q}$ where $p, q \in \N$ and  $gcd(p,q)=1$, then $\big(\frac{p}{q}\big)^{k_1}+\big(\frac{p}{q}\big)^{k_2}+\cdots +\big(\frac{p}{q}\big)^{k_n}=1$, i.e., $p^{k_1}q^{k-k_1}+p^{k_2}q^{k-k_2}+\cdots +p^{k_n}q^{k-k_n}=q^k$ where $k=k_1+k_2+\cdots +k_n$. If $p>1$, the left hand side is divisible by $p$ whereas the right hand side is not divisible by $p$ which is a contradiction and hence, $p=1$. 
		\end{proof}
		\begin{remark}
			Continuing with the setting of Theorem \ref{type III}, we will discuss more concretely regarding the  possible types of  $H^\infty$. 
			\begin{enumerate}
				\item Let $\dim \mathcal{H} = 2$ and let $\omega_1 = \frac{1}{3}, \omega_2 = \frac{2}{3}$. A little thought should convince the reader that, in this case, the group $\mathbb{G}$ generated by $\{\frac{1}{3}, \frac{2}{3}\}$ is $\R_+^*$ and hence, in this case, $H^\infty$ is of type $III_1$.\\
				\item 
				It is also possible that $H^\infty$ is a factor of type $III_\lambda$ for some irrational $\lambda \in (0, 1)$. Let us produce such an example. Note that the equation $x^2 + x -1 = 0$ has an irrational solution in $(0, 1)$, say, $\lambda$. If we let $\dim \mathcal{H} = 2$ and $\omega_1 = \lambda, \omega_2 = 1 - \lambda$, then the subgroup $\mathbb{G}$ generated by $\{\lambda, 1-\lambda(= \lambda^2)\}$ is clearly $\{\lambda^n: n \in \mathbb{Z} \}$ so that in this case $H^\infty$ is of type $III_\lambda$. \\
				\item  $H^\infty$ can never be of type ${III}_0$ in the case when $\mathcal{H}$ is finite-dimensional.  
			\end{enumerate}
		\end{remark}

	\section{Automorphism induced by second quantisation}
	In this section we deal with \textbf{Question $1$} stated in the introduction preceding the statement of Theorem C. We answer \textbf{Question $1$} in the affirmative in the case when $\mathcal{H}$ is finite-dimensional with $\dim \mathcal{H} > 1$, and $\omega$ is the constant sequence $\frac{1}{\dim \mathcal{H}}$.
	

	
	Given a unitary $U$ on $\mathcal{H}$, recall from Section $3$ the corresponding second quantization $\Gamma_U$ on $\mathcal{F}(\mathcal{H})$ and the associated automorphism $\widetilde{\Gamma}_U$ of $B(\mathcal{F}(\mathcal{H}))$. First note that $\widetilde{\Gamma}_U$ does not necessarily leave $H^{\infty}$ invariant and moreover, since the multiplication in the von Neumann algebra $H^{\infty}$ is different from that of $B(\mathcal{F}(\mathcal{H}))$, possibly one should not expect that the restriction of $\widetilde{\Gamma}_U$ to $H^{\infty}$ would induce an automorphism of $H^{\infty}$. To see this, let us consider the following simple example. Let $\omega = (\omega_i)_{i \in \Theta}$ denote a sequence of positive real numbers
	such that $\sum_{i \in \Theta} \omega_i = 1$
	and not all the $\omega_i$'s are equal. Let $i_0, j_0 \in \Theta$, $i_0 \neq j_0$, be such that $\omega_{i_0} \neq \omega_{j_0}$. Consider the unitary $U$ on $\mathcal{H}$ defined as follows: $$ U(e_{i_0}) = e_{j_0}, U(e_{j_0}) = e_{i_0} \ \text{and} \  U(e_t) = e_t, \ \text{for all} \ t \in \Theta \setminus \{i_0, j_0\}.$$ 
	It follows from Proposition \ref{multiplications}(iv) that
	$r_{i_0} \circ r_{j_0}^* = r_{i_0} r_{j_0}^* + \omega_{i_0} p_{\Omega} r_{j_0}^* l_{i_0}$ and hence, an appeal to Equation \eqref{eqquan} shows that
	$$\widetilde{\Gamma}_U \big(r_{i_0} \circ r_{j_0}^* \big) = \widetilde{\Gamma}_U \big( r_{i_0} r_{j_0}^* + \omega_{i_0} p_{\Omega} r_{j_0}^* l_{i_0} \big) = r_{j_0} r_{i_0}^* + \omega_{i_0} p_{\Omega} r_{i_0}^* l_{j_0}.$$
	On the other hand,
	$$ \widetilde{\Gamma}_U(r_{i_0}) \circ \widetilde{\Gamma}_U(r_{j_0}) = r_{j_0} \circ r_{i_0}^* = r_{j_0} r_{i_0}^* + \omega_{j_0} p_{\Omega} r_{i_0}^* l_{j_0}.$$ 
	Since $\omega_{i_o} \neq \omega_{j_0}$, it is clear that 
	$\widetilde{\Gamma}_U \big(r_{i_0} \circ r_{j_0}^* \big) \neq
	\widetilde{\Gamma}_U(r_{i_0}) \circ \widetilde{\Gamma}_U(r_{j_0})$. This example shows that if $\omega = (\omega_i)$ is not a constant sequence, then it is possible to construct a a unitary $U$ on $\mathcal{H}$ such that $\widetilde{\Gamma}_U$ fails to induce an automorphism of $H^{\infty}$. Thus arises the following natural question:

	\textit{What happens in the case when $\mathcal{H}$ is finite-dimensional with $\dim \mathcal{H} = n > 1$ and \  $\omega$ is the constant sequence $\frac{1}{n}$, that is, $\omega_1 = \omega_2 = \cdots = \omega_n = \frac{1}{n}?$}
	
	The main result of this section, namely, Theorem \ref{quan}, answers the above question in the affirmative by proving that if $\mathcal{H}$ is finite-dimensional with $\dim \mathcal{H} > 1$ and \  $\omega$ is the constant sequence $\frac{1}{\dim \mathcal{H}}$, then the restriction of $\widetilde{\Gamma}_U$ to $H^{\infty}$ is indeed an automorphism of $H^{\infty}$. We are grateful to the anonymous referee(s) for pointing out to us a simple proof of the result.
	
	\textit{Throughout the rest of this section, $\mathcal{H}$ denotes an $n$-dimensional Hilbert space, $n > 1$, with the orthonormal basis $\{e_1, e_2, \cdots, e_n\}$ and $\omega_1 = \omega_2 = \cdots = \omega_n = \frac{1}{n}$.}
	
	Given a unitary $U$ on $\mathcal{H}$, set $f_i = U(e_i), 1 \leq i \leq n$, and consider the Markov operator $P_\omega^\prime$ on $B \big(\mathcal{F}(\mathcal{H})\big)$ defined by $$P_\omega^\prime(x) = \frac{1}{n} \sum_{i = 1}^n l_{f_i}^* x l_{f_i}, x \in B \big(\mathcal{F}(\mathcal{H})\big)$$ and let $\widetilde{H^\infty_n}$ denote the Poisson boundary associated with $\Big(B \big(\mathcal{F}(\mathcal{H})\big), P_\omega^\prime \Big)$. It follows from the discussion in Section $3$ preceding Proposition \ref{multiplications} that $H^\infty_n \ni x \mapsto \widetilde{\Gamma}_U(x) \in \widetilde{H^\infty_n}$ is an isomorphism of von Neumann algebras. We now aim to show that $H^\infty_n = \widetilde{H^\infty_n}$. Let $\circ^\prime$ denote the multiplication of $\widetilde{H^\infty_n}$. For each $i, 1 \leq i \leq n$, express $f_i = \sum_{j = 1}^{n} u_{i j} e_j$. Then $(u_{ij})$ is a unitary matrix, and hence, $\sum_{k = 1}^n \overline{u_{ki}} u_{kj} = \delta_{i, j}$. Then, given any $x \in B \big(\mathcal{F}(\mathcal{H})\big)$, we see that
	$$P_\omega^\prime(x) = \frac{1}{n} \sum_{i = 1}^n l_{f_i}^* x l_{f_i} = \frac{1}{n} \sum_{i = 1}^n \Big(\sum_{j, k = 1}^n \overline{u_{ij}} u_{ik} l_{e_j}^* x l_{e_k}\Big) 
	 = \frac{1}{n} \sum_{j= 1}^n l_{e_j}^* x l_{e_j} = P_\omega(x).$$
	 It now follows immediately by an appeal to Equation \eqref{formu} that $\big(H^\infty_n, \circ \big) = \big(\widetilde{H^\infty_n}, \circ^\prime \big)$.	 
	We have thus shown that:
	\begin{coro}\label{final}
		With notations as above, $\widetilde{\Gamma}_U|_{H^\infty_n}$ is an automorphism of $H^\infty_n$.
	\end{coro}
	
	\begin{remark}\label{quanrmk}
		Let $\mathcal{C}$ denote the set $\{r_{\xi}: \xi \in \mathcal{H}\}$. We define an automorphism of $H^{\infty}_n$ to be $\mathcal{C}$-preserving if it maps the set $\mathcal{C}$ onto itself. Let $\mathbb{G}$ denote the subgroup of Aut($H^{\infty}_n)$, the automorphism group of $H^{\infty}_n$, consisting of all $\mathcal{C}$-preserving automorphisms. Also let $\mathcal{U}(\mathcal{H})$ denote the unitary group of $\mathcal{H}$. It follows from Corollary \ref{final} that for any $U \in \mathcal{U}(\mathcal{H})$, $\widetilde{\Gamma}_U|_{H^\infty_n}$ is an element of $\mathbb{G}$. We assert that the map $$\mathcal{U}(\mathcal{H}) \ni U \mapsto \widetilde{\Gamma}_U|_{H^\infty_n} \in \mathbb{G}$$ establishes an isomorphism of groups. One can easily see that this is an injective group homomorphism. We now prove surjectivity of the map. Let $\Psi \in \mathbb{G}$. Since $\Psi$ is $\mathcal{C}$-preserving, for each $\xi \in \mathcal{H}$, there exists unique $\eta \in \mathcal{H}$ such that $\Psi(r_{\xi}) = r_{\eta}$. This allows us to define a linear map $U : \mathcal{H} \rightarrow \mathcal{H}$ as follows: For $\xi \in \mathcal{H}$, define $U(\xi)$ to be the element of $\mathcal{H}$ such that $r_{U(\xi)} = \Psi(r_{\xi})$. To show that $U \in \mathcal{U}(\mathcal{H})$, it suffices to see that $\{U(e_i): 1 \leq i \leq n\}$ is an orthonormal basis of $\mathcal{H}$ where, recall that, $\{e_i: 1 \leq i \leq n\}$ is an orthonormal basis of $\mathcal{H}$. As $\Psi(r_{e_i}) = r_{U(e_i)}, 1 \leq i \leq n$, we have that $$\langle U(e_i), U(e_j)\rangle = \langle r_{U(e_i)}\Omega, r_{U(e_j)}\Omega\rangle = \langle \Psi(r_{e_i})\Omega, \Psi(r_{e_j})\Omega\rangle = \delta_{i, j}.$$ Hence, $U \in \mathcal{U}(\mathcal{H})$ and  we conclude that $\Psi = \widetilde{\Gamma}_U|_{H^\infty_n}$. Thus, \textit{$\mathcal{U}(\mathcal{H})$ and $\mathbb{G}$ are isomorphic as groups}, proving the assertion.
	\end{remark}
	
	Summarizing the foregoing discussions, we have the following theorem.
	\begin{theorem}\label{quan}
		Let $\mathcal{H}$ be finite-dimensional with $\dim \mathcal{H} > 1$, and let $\omega$ be the constant sequence $\frac{1}{\dim \mathcal{H}}$. 
		For each unitary $U$ on $\mathcal{H}$, $\widetilde{\Gamma}_U|_{H^\infty_{\dim \mathcal{H}}}$ is the unique automorphism of $H^\infty_{\dim \mathcal{H}}$ that takes $r_{\xi}$ to $r_{U \xi}$ for $\xi \in \mathcal{H}$. Further, the mapping $$\mathcal{U}(\mathcal{H}) \ni U \mapsto \widetilde{\Gamma}_U|_{H^\infty_{\dim \mathcal{H}}} \in \text{Aut}(H^{\infty}_{\dim \mathcal{H}})$$ of the unitary group $\mathcal{U}(\mathcal{H})$ of $\mathcal{H}$ to $\text{Aut}(H^{\infty}_{\dim \mathcal{H}})$, the automorphism group of $H^{\infty}_{\dim \mathcal{H}}$, is an injective group homomorphism.
	\end{theorem}

\subsection*{Acknowledgements} The authors are grateful to the anonymous referee(s) for an extremely careful reading of an earlier version of the manuscript and several insightful comments and suggestions, in particular, for making us aware of the study of the KMS states on the Cuntz algebra $\mathcal{O}_n$ ($n \in \mathbb{N}$) and their GNS representations, especially, for pointing out the references \cite{DEvans} and \cite{Iz4}, which has helped in substantial improvement of the presentation. The first author thanks J C Bose Fellowship of SERB (India) for financial support. 
P. Bikram acknowledges the support of the grant CEFIPRA-6101-1 and the fourth author is supported by the NBHM (India) post-doctoral fellowship.

\end{document}